\documentclass[english,12pt]{article}
\usepackage{amsmath}
\usepackage[a4paper, left=-0.5cm, right=-0.5cm]{geometry}
\usepackage{amsthm}
\usepackage{pdfsync}
\usepackage{a4wide}
\usepackage{amssymb}
\usepackage{amsfonts}
\usepackage{color}
\usepackage{empheq}
\usepackage{enumerate}
\usepackage{graphicx}
\pdfoutput=1
\usepackage{authblk}

\usepackage{fontenc}
\usepackage[utf8]{inputenc}
\usepackage{afterpage}
\usepackage[dvipsnames]{xcolor}
\usepackage{newtxtext,commath}
\usepackage{amscd, amsmath, amsthm, amssymb, amsfonts}
\usepackage{color}
\usepackage{enumerate}
\usepackage{graphicx}
\usepackage{ragged2e}
\usepackage{hyperref}
\usepackage{times}
\numberwithin{equation}{section}
\usepackage{comment}
\usepackage{amsfonts}
\usepackage{color}
\usepackage{enumerate}
\newtheorem{Th}{Theorem}

\newtheorem{prop}{Proposition}
\newtheorem{Def}{Definition}
\newtheorem{obs}{Remark}

\def\rr{\mathbb{R}}

\def\nn{\mathbb{N}}

\def\eps{\varepsilon}

\def\eps{\varepsilon}

\def\nn{\mathbb{N}}

\def\eps{\varepsilon}

\def\be{\begin{equation}}
	\def\ee{\end{equation}}

\numberwithin{equation}{section} \numberwithin{Th}{section}
\numberwithin{cor}{section} \numberwithin{lema}{section}
\numberwithin{prop}{section} \numberwithin{obs}{section}
\numberwithin{Def}{section}
\title{Sharp $H^2$-regularity in dimensions $N\geq 5$ and beyond  for  two  classes of elliptic problems with critical unbounded coefficients}
\author[1, 2, \thanks{Emails: \texttt{cristian.cazacu@fmi.unibuc.ro}, \texttt{ana.calina@math.unibuc.ro}}]{Cristian Cazacu}
\author[1, $^*$, \thanks{A.C. was partially supported by a doctoral fellowship offered by the Doctoral School of Mathematics, University of Bucharest and by the NSF-UEFISCDI grant ``Linear and nonlinear stability of physical flows," a complex bilateral research collaboration between U.S. and Romanian researchers, UEFISCDI project code ROSUA-2024-0001.}]{Adelina Călina}

\affil[1]{Faculty of Mathematics and Computer Science, University of Bucharest,
14 Academiei Street, 010014 Bucharest, Romania}
\affil[2]{Gheorghe Mihoc-Caius Iacob Institute of Mathematical Statistics and Applied Mathematics of the Romanian Academy,
13 Calea 13 Septembrie, Sector 5, 050711 Bucharest, Romania}

\date{}

\begin{document}
\maketitle
\textit{\textit{Keywords}: Hardy--Rellich inequalities; Elliptic regularity; Critical singular coefficients;\\
Regularized potentials and data; Uniform a priori estimates.}

\begin{abstract}
We establish sharp parameter thresholds governing $H^2$ -regularity of weak solutions in $H_0^1$  for two classes of  elliptic problems with critical unbounded perturbations, in dimensions  $N\geq 5$. More precisely,  we  consider two distinct  $\lambda$-parametric elliptic problems $ - \Delta v + \lambda \frac{x\cdot \nabla v}{|x|^{2}} =f$ and $-\Delta v 
+ \lambda \frac{v}{|x|^2}=f$ posed in a bounded $C^2$-domain $\Omega\subset \rr^N$ containing the origin $x=0$. We observe that the singular perturbations  $\frac{x\cdot \nabla v}{|x|^{2}}$ and $\frac{v}{|x|^2}$  are homogeneous operators of order 2 consistent with the scaling of the Laplacian.  
In view of the Hardy inequality the problems are well-posed in $H_0^1(\Omega)$ for $\lambda<\frac{N-2}{2}$ and $\lambda>-\frac{(N-2)^2}{4}$ respectively. 

The main results are as follows. For the first problem we show that any solution $v\in H_0^1(\Omega)$ belongs to $H^2(\Omega)$ for any $\lambda<
\frac{N-2}{2}$ provided $f\in L^2(\Omega)$. This fully extends the previous $H^2$ regularity properties obtained by Kim and Tsai in \cite{Kim-Tsai} for $\lambda\leq 0$. 
For the second problem we show that $H^2$ regularity holds for any $\lambda>- \frac{N(N-4)}{4}$ and fails for any $\lambda\in \left(-\frac{(N-2)^2}{4},-\frac{N(N-4)}{4}\right]$. This extends sharply the range of $\lambda
\in \left(-\frac{N(N-4)}{4}, \frac{N(N-4)}{4}\right)$ obtained when applying the Kato perturbation theory in \cite{Kato}.  
In addition, we develop sharp second order Hardy-Rellich type inequalities for the involved elliptic operators which are essential in the above proofs.

\end{abstract}

\section{Introduction} 
The most basic  regularity properties hold for distributional $L_{loc}^1$ harmonic functions  which  are known to be analytic.
The weak regularity theory of elliptic operators with bounded coefficients is in general well understood by now (see for instance the recent monograph by Fern\^{a}ndez-Real and Ros-Oton \cite{Real-Ros-Oton} and also the classical book by Gilbarg and Trudinger  \cite{Xavier}).   
    This is a flourish subject which has been extensively developed in the last decades, facing very interesting and deep results. Motivated by real world applications, as in \cite{Adams}    Sobolev spaces constitute the natural functional framework used to introduce the variational approach  of the boundary value problems. A fundamental result concerning weak regularity in this setting is the Calder\'on-Zigmund theory which asserts that if $ \Omega \subset \mathbb{R}^{N}$ is a bounded domain with $ C^{2}$ boundary and $ f \in L^{2}(\Omega)$ then the $H_0^1(\Omega)$ weak solution to the Dirichlet problem
\begin{align}
    \left\{\begin{array}{ll} 
-\Delta v = f(x),  &		\textrm{ in} \quad  \Omega\\
	v=0, &  \textrm{ on }  \partial \Omega,\\ 
	\end{array}\right.
\end{align}
gains two  partial derivatives in the weak sense with respect to $f$, that is $ v \in H^{2}(\Omega)$. Iteratively, $f\in H^k(\Omega)$ implies $v\in H^{k+2}(\Omega)$ and therefore $f\in C^\infty(\Omega)$ implies $v\in C^\infty(\Omega)$.  
These weak regularity results extend to more general elliptic equations written in divergence form $ - \partial_{i}(a_{ij}\partial_{j}v) + b_{i}\partial_{i}v + cv = f$, 
where coefficients $ b_{i}$ and $ c $ are bounded, whereas   $ a_{ij} \in C^{1}(\Omega)$ and satisfy  the  ellipticity condition (i.e. there exists a constant $ C > 0$ such that  $a_{ij}(x)\xi_{i}\xi_{j} \ge C |\xi|^{2},$ for all $ \xi \in \mathbb{R}^{N}$ and a.e $ x \in \Omega$). The smoothness and the boundedness assumptions of the coefficients crucially influence the sharp weak regularity.  

These features might change dramatically when the elliptic operator has singular coefficients, in which cases the standard elliptic regularity (cf. e.g. \cite{anderson}, \cite{Charro}, \cite{CalderonZygmund1952}, \cite{CalderonZygmund1956}, \cite{Stampacchia}) does not apply. In order to understand this phenomena it is more appropriate to consider the paradigmatic case of the Laplacian with unbounded perturbations. 

In this paper we  focus to complete the study of sharp weak regularity for two well-known $\lambda$-parametric  Dirichlet boundary value problems with critical singular pertubations of the Laplacian in \emph{bounded domains}, namely 
\begin{align}\label{p1}
    \left\{\begin{array}{ll} 
-\Delta v+\lambda \frac{x\cdot \nabla v}{|x|^2} = f,  &		\textrm{ in} \quad  \Omega\\
	v=0, &  \textrm{ on }  \partial \Omega,\\ 
	\end{array}\right.\tag{P1}
\end{align}
and
\begin{align}\label{p2}
    \left\{\begin{array}{ll} 
-\Delta v+\lambda \frac{v}{|x|^2} = f,  &		\textrm{ in} \quad  \Omega\\
	v=0, &  \textrm{ on }  \partial \Omega,\\ 
	\end{array}\right.\tag{P2}
\end{align}
with $\lambda\in \rr$ and $0\in \Omega$. Definitely,  problems \eqref{p1}-\eqref{p2} have unbounded coefficients because the presence of the singularity at the origin. Moreover both singular operator perturbations $\lambda \frac{x\cdot \nabla v}{|x|^2}$ and $\lambda \frac{v}{|x|^2}$ of problems \eqref{p1} and \eqref{p2}, respectively,  exhibit critical homogeneity of degree 2, the same as  the Laplacian,  which makes the analysis more intricate. In that respect, a fundamental role in this context is played by the classical Hardy inequality 
   in dimensions $N\geq 3$  
   \begin{equation}\label{HI}
       \int_{\Omega} |\nabla v|^2 dx \geq \frac{(N-2)^2}{4} \int_{\Omega}\frac{v^2}{|x|^2}, \quad \forall v\in H_0^1(\Omega).
   \end{equation}
It is well-known that $\lambda_\star:=\frac{(N-2)^2}{4}$ is the optimal constant in \eqref{HI} which separates the regimes of well-posedness and instability for problems \eqref{p1}-\eqref{p2}. For our purpose, next we state the definitions of weak solutions which could also extend to the whole space $ \mathbb{R}^{N}$. 

\begin{Def}\label{def1}
 Let $ f \in L^{2}(\Omega)$. A function $ v \in H_{0}^{1}(\Omega)$ is a weak solution to \eqref{p1} if 
\begin{align}\label{20}
    \int_{\Omega}\left(\nabla v \nabla \phi + \lambda \frac{x\cdot \nabla v}{|x|^{2}}\phi\right) dx = 
    \int_{\Omega} f\phi dx,  \quad \textrm{for all $ \phi \in H^{1}_{0}(\Omega)$} 
\end{align}
\end{Def}

\begin{Def}\label{def2}
 Let $ f \in L^{2}(\Omega)$. A function $ v \in H_{0}^{1}(\Omega)$ is a weak solution to \eqref{p2} if 
\begin{align}\label{20_b}
    \int_{\Omega}\left(\nabla v \nabla \phi +\lambda \frac{  v \phi}{|x|^{2}}\right) dx = 
    \int_{\Omega} f\phi dx,  \quad \textrm{for all $ \phi \in H^{1}_{0}(\Omega)$} 
\end{align}
\end{Def}

\paragraph{\bf Known regularity results on the whole space $\mathbb{R}^{N}$.}
Some regularity aspects have been already treated for the previous problems.
Among the classical contributions to the study of singular elliptic operators we refer to the work
of Kato \cite{Kato}, which provides a characterization of the domain of elliptic operators with
singular, relatively bounded perturbations in $\mathbb{R}^{N}$ . This result has served as a starting point for
subsequent investigations on weak regularity properties of singular operators.
The well-known Hardy-Rellich inequality  (see e.g. \cite{Beckner2008}, \cite{Cazacu3}, \cite{Gesztesy} )
\begin{align}\label{104}
\frac{N^{2}}{4}\int_{\mathbb{R}^{N}}
\frac{|x\nabla v|^{2}}{|x|^{4}}\,dx
&\le
\int_{\mathbb{R}^{N}}|\Delta v|^{2}\,dx,
\quad
\forall\, v\in C_c^\infty(\mathbb{R}^{N}),\; N\ge 3.
\end{align}
ensures that the singular term $\frac{x\nabla v}{|x|^{2}} $ is relatively bounded with respect to $ -\Delta$ when $ \lambda \in (-\frac{N}{2},\frac{N}{2})$.
Hence,   Kato's result guarantees that the domain of the operator $ -\Delta + \frac{x\nabla \cdot}{|x|^{2}}$ coincides with $ H^{2}(\mathbb{R}^{N})$ for all $\lambda \in (-\frac{N}{2},\frac{N}{2}) $ and $ N \geq 3$ .
The famous Rellich inequality  (see  e.g. \cite{Metafune2015}  \cite{Cazacu1}) states that
\begin{align}\label{105}
\frac{N^{2}(N-4)^{2}}{16}\int_{\mathbb{R}^{N}}
\frac{| v|^{2}}{|x|^{4}}\,dx
&\le
\int_{\mathbb{R}^{N}}|\Delta v|^{2}\,dx,
\quad
\forall\, v\in C_c^\infty(\mathbb{R}^{N}),\; N\ge 5.
\end{align}
which implies that the singular potential which arises in  problem \eqref{p2} is relatively bounded with respect to $ - \Delta$ when $ \lambda \in (-\frac{N(N-4)}{4},\frac{N(N-4)}{4})$. Since the Laplacian is a closed operator, it follows from Kato's theorem that the weak solution of problem \eqref{p2} belongs to $ H^{2}(\mathbb{R}^{N}) $ for all $ \lambda \in (-\frac{N(N-4)}{4},\frac{N(N-4)}{4})$ and $ N \geq 5$.
However, Kato's theorem does not apply to \emph{bounded domains} because  inequalities \eqref{104} and \eqref{105} make sense for  functions in $ C_{c}^{\infty}(\Omega)$ which is not dense in $ H^{2}(\Omega)\cap H_{0}^{1}(\Omega)$. In contrast, when $ N \geq 5$, the space $ C_{c}^{\infty}(\Omega)$ is dense in $ H_{0}^{2}(\Omega)$, allowing these inequalities to be extended by density to $ H_{0}^{2}(\Omega)$, but not to the full domain of the Laplacian. 

A comprehensive study of the domains of this class of singular operators in the whole space $ \mathbb{R}^{N}$ was carried out by Metafune-Okazawa-Sobajima-Spina \cite{Metafune2016}. Going beyond Kato's perturbation framework, the authors established spectral theoretic techniques to characterize the domains of a broad class of scale-invariant singular elliptic operators. In particular, they considerably enlarged the admissible range of the parameter $ \lambda$ for which a complete domain characterization can be obtained. Hence, they prove that $D_{\min}\left(-\Delta+\frac{\lambda}{|x|^{2}}\right)
=\{v\in H^{2}(\mathbb{R}^{N})
\mid\frac{v}{|x|^{2}}\in L^{2}(\mathbb{R}^{N})\}$ for every $ \lambda \in \left(-\frac{N(N-4)}{4}, \infty\right)$. 
Also they get that $D_{\min}\left(-\Delta+\lambda\frac{x\nabla}{|x|^{2}}\right)
=\{v\in H^{2}(\mathbb{R}^{N})
\mid \frac{x\nabla v}{|x|^{2}}\in L^{2}(\mathbb{R}^{N})\}$ for every $ \lambda  \in \left(-\infty, \frac{N}{2}\right)$ (see \cite[Theorem 4.5]{Metafune2016}). Their characterization is formulated in terms of $ H_{loc}^{2}(\mathbb{R}^{N} \setminus \{0\})$, together with suitable weighted integrability conditions describing the behavior near the singularity. Moreover, their results are obtained without explicit restrictions on the space dimension N. Thus, the techniques employed in \cite{Kato}, \cite{Metafune2016} are intrinsically tied to the whole-space setting and their extension to bounded domains remains a challenging issue. It is worth emphasizing that both Kato's theory and the approach adopted by Metafune are developed independently of the variational formulation of weak solutions. 

\paragraph{\bf Known regularity results in bounded domains.} A significant contribution in the bounded domain setting is due to Peral-Soria \cite{Ireneo Peral}, who investigated the regularizing effects of the Hardy-type singular potential in \eqref{p2}. Assuming $ f \in L^{m}(\Omega)$, they established conditions under which Calderon- Zygmund- Stampacchia regularity estimates remain valid. More precisely, if $ 1< m< \frac{2N}{N+2}$ and $ \lambda > \frac{N(1-m)(N-2m)}{m^{2}}$ then the weak solution gains regularity in $ W_{0}^{1, m^{\ast}}(\Omega)$. Moreover, if $ \frac{2N}{N+2}< m< \frac{N}{2}$ and $ \lambda > \frac{N(1-m)(N-2m)}{m^{2}}$, they proved that the weak solution belongs to $ L^{m^{\ast \ast}}(\Omega)$, where $ m^{\ast \ast}= \frac{Nm}{N-2m}$. 

These results also show that the Hardy-Leray operator does not exhibit the classical regularity results: even for $ f \in L^{\infty}(\Omega)$, the corresponding solution need not to be bounded. Moreover, the estimates given by Peral-Soria do not, in general, imply $ H^{2}-$ regularity. Indeed, for $ f \in L^{2}(\Omega)$ they only obtain $ u \in L^{2^{\ast \ast}}(\Omega)$, while $ H^{2}(\Omega)$ is not guaranteed. 

In the same spirit, Petitta-Leonori in \cite{LP07} analyzed the regularizing effect of the drift term in problem \eqref{p1}, assuming $ f \in L^{m}(\Omega)$. They showed that the classical regularity results are preserved under suitable assumptions on $ \lambda$. 
In the subcritical regime $ 1< m< \frac{2N}{N+2}$, if $ \lambda > \frac{N(1-m)}{m}$, the corresponding weak solution belongs to $ W_{0}^{1,m^{\ast}}(\Omega)$. 
For $ \frac{2N}{N+2}< m< \frac{N}{2}$, under the same condition on $ \lambda$, they further proved that the solution improves the integrability, leading to $ L^{m^{\ast \ast}}(\Omega)$. 
A striking difference with respect to the Hardy-Leray problem appears in the supercritical case $ m > \frac{N}{2}$, where the solution enjoys an $ L^{\infty}$-regularizing effect. 

In the bounded domain setting, Kim and Tsai in \cite{Kim-Tsai} established an $ H^{2}$-regularity result for weak solutions to the more general elliptic equation $ -\Delta v-b\cdot v= f.$ Their approach relies on Lorentz space estimates and requires $ N \geq 5$. More precisely, assuming $ b \in L^{N,\infty},\quad \operatorname{div} b \in L^{\frac N2,\infty}(\Omega) $  and $\operatorname{div} b \geq0,  $  they proved that weak solutions belong to $ H^{2}(\Omega)$. 

Since problem \eqref{p1} corresponds to the vector field $ b(x)= -\frac{\lambda x}{|x|^{2}}$, their result applies whenever $ \lambda \leq 0$, yielding $ H^{2}(\Omega)$-regularity for weak solutions in dimensions $ N \geq 5$ in these cases.

\paragraph{\bf Main results for problem \eqref{p1}.}
  The existence and uniqueness of $H_0^1(\Omega)$ weak solutions  in \eqref{20}  hold for any $\lambda<\frac{N-2}{2}$. This follows from Lax-Milgram lemma due to the coercivity  of the functional energy: 
  \begin{multline}\label{energy_funct}
l_\lambda[\phi]:=\int_{\Omega} |\nabla \phi|^2 +\lambda \frac{x\cdot \nabla \phi}{|x|^2}\phi dx \\
  =\int_{\Omega} |\nabla \phi|^2 dx -\lambda \frac{N-2}{2}\int_{\Omega}\frac{\phi^2}{|x|^2} dx \geq \left(1-\frac{2\lambda}{N-2}\right)\int_{\Omega} |\nabla \phi|^2 dx,\qquad \qquad
  \end{multline}
  where the last inequality leads from \eqref{HI}. At the critical and supercritical levels $\lambda\geq \frac{N-2}{2}$, the situation becomes substantially more delicate, with loss of coercivity and the emergence of singular phenomena as we will emphasize later.

To control the lower-order terms arising from the drift, Metafune employed Hardy-Rellich type estimates involving an $ \varepsilon$-splitting \cite[Lemma 2.4]{Metafune2016}. Here, we show that, under a suitable restriction on the parameter $ \lambda$, the interpolation term can be removed, yielding the following parametric Hardy-Rellich inequality. This ingredient will be very  useful to obtain the sharp regularity for \eqref{p1}. This result is interesting in itself and its significance extends beyond that. 

\begin{Th}\label{HR_lambda}
   Let $N\geq 3$. For any $N\neq 4$ and any $\lambda \in \rr\setminus\left\{\lambda_n:=\frac{(2n+N-2)^2-4} {2(N-4)} \  | n \in \nn^\star \right\} \setminus \{\frac{N}{2}\}$ or $N=4$ and any $\lambda\in \rr\setminus\{\frac{N}{2}\}$  there exists $C(\lambda, N)>0$ such that  
   \begin{equation} 
\int_{\rr^N}  \left | -
\Delta v + \lambda\frac{x
\cdot \nabla v}{|x|^2} 
\right|^2  dx \geq C(\lambda, N) \int_{\rr^N} \frac{|x\cdot \nabla v|^2}{|x|^4} dx, \quad \forall v\in C_c^\infty(\rr^N). \label{HR_ineq1}
   \end{equation}
The constant $C(\lambda, N)$   is explicitly given by $C(\lambda, N) =$
  \begin{equation}
  =\left\{
\begin{array}{ll}
    (2-\lambda)^2, & N=4, \\[3pt]
    \left(\lambda-\frac{3}{2}\right)^2, & N=3 \textrm{ and } \lambda \geq -\frac{1}{2} \\ [3pt]
   \min\left\{\left(\lambda-\frac{3}{2}\right)^2, 4\min\limits_{k=1, k_0(\lambda, 3)} \left[k(k+1) +\frac{1}{2}\left(\lambda-\frac{3}{2}\right)\right]^2\right\},  & N=3 \textrm{ and } \lambda< -\frac{1}{2}\\ [3pt]
   \left(\frac{N}{2}-\lambda\right)^2, & \lambda \leq \frac{N^2-2N-2}{2(N-4)} ,  N>4. \\ [3pt]
      \min\left\{\left(\frac{N}{2}-\lambda\right)^2, \left(\frac{2}{N-4}\right)^2\min\limits_{k=1, k_0(N, \lambda)} \left[k(k+N-2) +\frac{N-4}{2}\left(\frac{N}{2}-\lambda\right)\right]^2\right\}, & \lambda > \frac{N^2-2N-2}{2(N-4)} ,  N>4.
\end{array}\right. \label{sharp_const}
  \end{equation} 
  where 
  \begin{equation}
    k_0(N,\lambda)=\left\{\begin{array}{cc}
     \frac{-N+\sqrt{4\lambda(N-4)-N^2+4N +4}}{2},    &  \textrm{ if }\  \frac{-N+\sqrt{4\lambda(N-4)-N^2+4N +4}}{2} \in \nn \\
\left[\frac{-N+\sqrt{4\lambda(N-4)-N^2+4N +4}}{2}
\right]+1,         & \textrm{otherwise}.
    \end{array}\right. \label{k0}
\end{equation} 
\end{Th}
The excluded countable set of $\lambda$'s in Thm. \ref{Teorema 1} is needed to ensure the positivity of $C(\lambda, N)$. However, this set could be slightly refined with a demanding technical effort. Since this would be  irrelevant for the main regularity results, for simplicity we prefer to keep  Theorem \ref{Teorema 1} in the actual form.
The main regularity result is as follows. 

\begin{Th} \label{Teorema 1}
   Assume $N\geq 3$ and $\Omega \subset \mathbb{R}^{N}$ be a bounded domain with $ C^{2}$ boundary such that $0\in \Omega$. Let  $v$ be the weak solution to  problem \eqref{p1} in the sense of Definition \ref{def1}. \\
  (a)  If $N\geq 5$ then $ v \in H^{2}(\Omega)$ for any $\lambda <\frac{N-2}{2}$. \\
    (b)  Otherwise,  if $N\geq 3$ and $\lambda\in \left(\frac{N-2}{2}, \frac{N}{2}\right]$ the $H^2$-regularity of $v$ fails to be true. \\
    (c) If $ N \geq 5$ and $ \lambda \in [\frac{N^{2}}{2(N-2)},\frac{N^{2}-4}{2(N-4)}]$ the $H^{2}$-regularity of v fails to be true. 
In addition, if $ N=4$ and $ \lambda \in[4,\infty)$ or $ N = 3$ and $ \lambda \in (-\infty,-\frac{5}{2}] \cup [\frac{9}{2}, \infty)$ the $ H^{2}$-regularity of v fails to be true. 
\end{Th}

This result was proved by Kim and Tsai in \cite{Kim-Tsai} for $\lambda <0$ in dimensions $ N\geq 5$. Our new contribution in Th. \ref{Teorema 1} arise in the cases $\lambda\in (0, \frac{N-2}{2})$. In this range of $\lambda$ the sharp $H^2$ regularity is the most surprising because the functional energy \eqref{energy_funct} becomes  increasingly less coercive as $\lambda$ approaches $\frac{N-2}{2}$.

\paragraph{\bf Main results for problem \eqref{p2}.} Definitely, the Hardy inequality \eqref{HI} ensures existence and uniqueness 
of weak solution to \eqref{p2} in the sense of Def. \eqref{def2} for any $\lambda>- \frac{(N-2)^2}{4}$ in which cases the associated quadratic form is coercive:   
\begin{equation}
    \int_{\Omega} |\nabla \phi|^2+\lambda \frac{\phi^2}{|x|^2} dx \geq \left(1+\min\{\lambda,0 \}\frac{4}{(N-2)^2}\right)\int_{\Omega} |\nabla \phi|^2 dx, \quad \forall \phi \in H_0^1(\Omega). 
\end{equation}

To obtain the main regularity results we need a parametric Rellich type inequality. Although the optimal characterization of the admissible values of the parameter $ \lambda$ was established by Metafune et al. \cite{Metafune2015} using spectral methods, we present a direct proof of a sufficient condition based on the spherical harmonic decomposition. This condition is completely adequate for our purposes and, in dimension $ N \geq 6$, yields  a larger admissible range for $\lambda$ than the sufficient condition stated in \cite[Proposition 2.2]{Metafune2015}.

 \begin{Th}\label{Ineq_lambda}
   Let $N\geq 5$. For any $\lambda \in (\frac{-N^{2}+2N+2}{4},\infty)\setminus\{-\frac{N(N-4)}{4}\}$ there exists $C(\lambda, N)>0$ such that  
   \begin{equation}
\int_{\rr^N}  \left | -
\Delta v + \lambda\frac{ v}{|x|^2} 
\right|^2  dx \geq c_2(\lambda, N) \int_{\rr^N} \frac{|v|^2}{|x|^4} dx, \quad \forall v\in C_c^\infty(\rr^N), \label{HR_ineq2}
   \end{equation}
   where 
   $$c_2(\lambda, N):=\left[\left(\frac{N^{2}}{4} +2\lambda\right)\left(\frac{N-4}{2}\right)^{2} + \lambda^{2} + 2 \lambda(N-4)\right].$$
\end{Th}  
This result extends also the parameter range given by Vișan  in \cite[Proposition 3.2]{{Killip2018}}.

Next we state the main regularity result for problem \eqref{p2}. 

\begin{Th} \label{Teorema 2}
   Assume $N\geq 3$ and $\Omega \subset \mathbb{R}^{N}$ be a bounded domain with $ C^{2}$ boundary such that $0\in \Omega$. Let  $v$ be the weak solution to  problem \eqref{p2} in the sense of Definition \ref{def2}. \\
  (a)  If $N\geq 5$ then $ v \in H^{2}(\Omega)$ for any $\lambda > \left(-\frac{N(N-4)}{4}, \infty\right)$. \\
    (b)  Otherwise,  if $N\geq 3$ and $\lambda\in \left(-\frac{(N-2)^2}{4}, -\frac{N(N-4)}{4}\right]$ the $H^2$-regularity of $v$ fails to be true.  
\end{Th}

These findings contribute to a more complete understanding of the regularity theory for elliptic equations with critical singular coefficients, and provide a unified framework bringing perturbative methods and sharp threshold phenomena.

The paper is organized as follows. In Section 2, we prove the $\lambda$-parametric Hardy–Rellich-type inequality stated in Theorem \ref{HR_lambda}, together with the Rellich inequality in Theorem \ref{Ineq_lambda}. These inequalities play a crucial role in establishing the $H^2$-regularity results in Theorems \ref{Teorema 1}–\ref{Teorema 2}. In Section \ref{section3}, we construct examples that illustrate and justify assertions $(b)$ and $(c)$ of Theorem \ref{Teorema 1}. In particular, we provide an example showing that $H^2$-regularity for problem \eqref{p1} fails in bounded domains when $\lambda > \frac{N-2}{2}$. This is in contrast with the whole-space setting $\mathbb{R}^N$, where Kato's result shows that $\lambda \in \left(-\frac{N}{2},\frac{N}{2}\right)$ ensures standard elliptic $H^2$-regularity. Finally, in the last section, we formulate maximum principles that are instrumental in the proofs of our main results. 

\textbf{The main ideas of the proof of Theorem\eqref{Teorema 1}}

\textbf{Step 1.} We consider an approximation Dirichlet problem 
\begin{align}
  \left\{\begin{array}{ll} 
-\Delta v_{\varepsilon} + \lambda \frac{x\nabla v_{\varepsilon}}{|x|^{2}}= f_{\varepsilon}(x),  &		x\textrm{ in} \quad \Omega\\
v_{\varepsilon}(x)=	0, & x \textrm{ on }  \partial \Omega,\\ 
	\end{array}\right.
\end{align}
where $ f_{\varepsilon}$ is a suitable truncation of the source $ f \in L^{2}(\Omega)$ with the property that $ f_{\varepsilon}$ vanishes near the origin. 

\textbf{Step 2.}
By Lax -Milgram theorem and standard elliptic regularity we prove that $ v_{\varepsilon} \in H_{0}^{1}(\Omega) \cap   C^{\infty}(B_{\varepsilon}\setminus \{0\})  $, where $ v_{\varepsilon}$ satisfies the homogeneous equation $-\Delta v_{\varepsilon} + \lambda \frac{x\nabla v_{\varepsilon}}{|x|^{2}} =0 $ near the origin.

\textbf{Step 3.} Using the maximum principle we compare $ v_{\varepsilon}$ with the radial solutions of the previous homogeneous equation and  we show that $ v_{\varepsilon} \in L^{\infty}(B_{\delta}(0)),$ where $ \delta < \varepsilon$.

\textbf{Step 4.} We recast the equation with drift term in this divergence form $ - |x|^{\lambda}\mathrm{div}(|x|^{-\lambda}\nabla v_{\varepsilon}) = 0$ and employing a cut-off argument we prove that $ v_{\varepsilon} \in H^{2}(\Omega)$, for any $ \varepsilon >0$. 

\textbf{Step 5.} Hardy-Rellich inequality \eqref{HR_ineq1} involves that the integral $\int_{\Omega} \frac{|x\nabla v_{\varepsilon}|^{2}}{|x|^{4}}dx $ is uniformly bounded in $ \varepsilon$ and this ensures that $ v_{\varepsilon}$ is uniformly bounded in $ H^{2}(\Omega)$. By standard arguments this implies that $ v \in H^{2}(\Omega)$.  

The proof of the Theorem \ref{Teorema 2} is somehow similar. As in the case of problem \ref{p2} we cannot prove that $ v_{\varepsilon}$ is bounded near the singularity. For this reason, we will directly compare $ v_{\varepsilon}$ with the radial solution to the homogeneous equation and showing that integral $ \int_{\Omega}\frac{v^{2}_\varepsilon}{|x|^{4}}dx $ is finite we establish that $ v_{\varepsilon} \in H^{2}(\Omega)$.  This approach can also be applied in the proof of \eqref{Teorema 1}. However, we have chosen to use an argument based on writing the equation with the drift term in divergence form.

\section{Parametric Hardy and Rellich type inequalities}

To establish Hardy Rellich type inequality, we shall make use of the following proposition, whose proof yields a complete algebraic characterization of the sequence of parameters $ \lambda_{n}$ via a generalized Pell equation. 

\begin{prop}\label{prop 1.1} There exists a sequence $ \{\lambda_{n}\}_{n\in \mathbb{N}} \subset \left(\frac{N^2-2N-2}{2(N-4)},\infty\right)$ such that  
\begin{equation}\label{claim}
  \min_{k=1, k_0(N, \lambda_{n})} \left[k(k+N-2) +\frac{N-4}{2}\left(\frac{N}{2}-\lambda_{n}\right)\right]^{2}=0, 
\end{equation}
where $ k_{0}(N,\lambda)$ is defined in \eqref{k0}.

Conversely, if \eqref{claim} holds for some parameter  $\lambda$ then there exists $ n \leq k_{0}(N,\lambda)$ such that $$ \lambda =\frac{(2n+N-2)^2-4} {2(N-4)}. $$
    \end{prop}
 
\begin{proof}[Proof of Theorem \eqref{HR_lambda}]
We first have 
\begin{align}
    \|L_{\lambda}\|_{L^{2}(\mathbb{R}^{n})}^{2}&= \int_{\rr^N}  \left | -
\Delta v + \lambda\frac{x
\cdot \nabla v}{|x|^2} 
\right|^2  dx\nonumber\\
    & =\int_{\mathbb{R}^{N}}|\Delta v |^{2}dx + \lambda^{2}\int_{\mathbb{R}^{N}}\left|\frac{x\cdot \nabla v }{|x|^{2}}\right|^{2} dx + 2 \lambda \int_{\mathbb{R}^{N}}- \Delta{v }\frac{x\cdot \nabla v }{|x|^{2}} dx.  \label{mixt}
\end{align}

Next we compute separately the mixt term $ I := \int_{\mathbb{R}^{n}}-\Delta v \frac{x\cdot \nabla v }{|x|^{2}}dx $ in \eqref{mixt}. Using Einstein's summation convention by integration by parts we successively obtain  
\begin{align}
     I =& \int_{\mathbb{R}^{N}}-\Delta v \frac{x\cdot \nabla v }{|x|^{2}}dx = - \int_{\mathbb{R}^{N}}v_{x_{i}x_{i}}\frac{x_{j}v_{x_{j}}}{|x|^{2}}dx = \int_{\mathbb{R}^{N}}v_{x_{i}}\left(\frac{x_{j}v_{x_{j}}}{|x|^{2}}\right)_{x_{i}} \nonumber \\
     & =\int_{\mathbb{R}^{N}}v_{x_{i}}\left(\delta_{ij}v_{x_j}|x|^{-2}+ x_{j}v_{x_{j}x_{i}}|x|^{-2} -2 |x|^{-4}x_{j}v_{x_{j}}x_{i}\right) dx \nonumber \\
     & = \int_{\mathbb{R}^{N}}\frac{|\nabla v |^{2}}{|x|^{2}} dx + \int_{\mathbb{R}^{N}}\left(v_{x_{i}}x_{j}v_{x_{j}x_{i}}|x|^{-2}  -2|x|^{-4}x_{j}v_{x_{j}}x_{i}v_{x_{i}} \right)dx \nonumber \\
     &= \int_{\mathbb{R}^{N}}\frac{|\nabla v |^{2}}{|x|^{2}}dx + \frac{1}{2}\int_{\mathbb{R}^{N}}x_{j}(v^{2}_{x_{i}})_{x_{j}}|x|^{-2}dx - 2 \int_{\mathbb{R}^{n}}\frac{|x\cdot \nabla v |^{2}}{|x|^{4}} dx \nonumber \\
     & = \int_{\mathbb{R}^{N}}\frac{|\nabla v |^{2}}{|x|^{2}}dx - \frac{1}{2}\int_{\mathbb{R}^{N}}v^{2}_{x_{i}}
     \left(x_{j}|x|^{-2}
     \right)_{x_{j}}dx-2\int_{\mathbb{R}^{N}}\frac{|x\cdot \nabla v |^{2}}{|x|^{4}}dx \nonumber \\
     & = \int_{\mathbb{R}^{N}}\frac{|\nabla v |^{2}}{|x|^{2}}dx -\frac{N-2}{2}\int_{\mathbb{R}^{N}}\frac{|\nabla v |^{2}}{|x|^{2}}dx - 2\int_{\mathbb{R}^{N}}\frac{|x\cdot \nabla v |^{2}}{|x|^{4}}dx \nonumber \\
     & = \frac{4-N}{2}\int_{\mathbb{R}^{N}}\frac{|\nabla v |^{2}}{|x|^{2}}dx - 2 \int_{\mathbb{R}^{N}}\frac{|x\cdot \nabla v |^{2}}{|x|^{4}}dx.  \label{mixt_final}
\end{align}
Combining \eqref{mixt} and \eqref{mixt_final}  we get that 
\begin{align}\label{eq50}
      |L_{\lambda}v |^{2}_{L^{2}(\mathbb{R}^{N})} =\int_{\mathbb{R}^{N}}|\Delta v |^{2}dx +(\lambda^{2} -4 \lambda)\int_{\mathbb{R}^{N}}\frac{|x \cdot \nabla v |^{2}}{|x|^{4}}dx + \lambda(4-N)\int_{\mathbb{R}^{N}}\frac{|\nabla v |^{2}}{|x|^{2}}dx. 
\end{align}
Next  we use spherical harmonics decomposition to expand $ v  $ as 
$$ v (x) = v (r\sigma) = \sum_{k=0}^{\infty}v_k(r)\phi_{k}(\sigma),$$
where the family $ \{\phi_{k}\}_{k \ge 0}$ is an othonormal basis in $ L^{2}(S^{n-1})$ formed by spherical harmonic functions $ \phi_{k}$ of degree k. In fact, $ \phi_{k}$ are smooth eigenfunctions of the Laplace Beltrami operator $ \Delta_{S^{n-1}}$ with the corresponding eigenvalues $ c_{k} = k(k+N-2)$ as in \cite[Theorem C.4.1]{Dupaigne}). In fact, we have the following properties  
\begin{align}
    \left\{\begin{array}{ll} 
-\Delta_{S^{n-1}}\phi_{k}  = c_{k}\phi_{k},  &	\textrm{ on }   S^{n-1}\\[3pt]
	\int_{S^{n-1}}\nabla_{S^{n-1}}\phi_{k}\cdot \nabla_{S^{n-1}}\phi_{l}d\sigma =c_{k}\int_{S^{n-1}}\phi_{k}\phi_{l}d\sigma = c_{k}\delta_{lk}, &    k, l \in \mathbb{N},\\ 
	\end{array}\right.
\end{align}
where $ \delta_{lk}$ represents the Kronecker symbol. 

Taking into account that $ \Delta = \partial_{rr}^{2}+ \frac{N-1}{r}\partial_{r} + \Delta_{S^{n-1}}$ we get that (for details, see e.g. \cite{Cazacu3})
\begin{align}\label{eq20}
    \int_{\mathbb{R}^{n}}|\Delta v |^{2}dx &= \sum_{k=0}^{\infty}
    \Bigg(\int^{\infty}_{0}r^{N-1}|v^{
    \prime\prime
    }_{k}|^{2}dr + (N-1 + 2 c_{k})\int^{\infty}_{0}r^{N-3}|v_k^\prime|^{2}dr \nonumber \\
    & + \left(c_{k}^{2}+ 2 c_{k}(N-4)\right)\int^{\infty}_{0}r^{N-5}v_k^{2}dr \Bigg). 
\end{align}
Similarly, we have 
\begin{equation}
     \int_{\mathbb{R}^{N}}\frac{|\nabla v |^{2}}{|x|^{2}} dx = \sum^{\infty}_{k=0}\left(\int^{\infty}_{0}r^{N-3}|v_k^\prime|^{2}dr + c_{k}\int_{0}^{\infty}r^{N-5}v_k^{2}dr\right) \label{form2}
\end{equation} 
and 
\begin{equation}
\int_{\mathbb{R}^{N}}\frac{|x\cdot \nabla v|^{2}}{|x|^{4}}dx = \sum^{\infty}_{k=0} \int^{\infty}_{0}r^{N-3}|v_k^\prime|^{2}dr. \label{form3}
\end{equation}

In view of \eqref{eq20}-\eqref{form3} identity \eqref{eq50} becomes 
\begin{align}
    \|L_\lambda v\|_{L^2(\rr^N)}^2  &= \sum_{k=0}^{\infty} \Bigg \{ \int^{\infty}_{0} r^{N-1}|v_k^{\prime\prime}|^{2}dr \nonumber\\
    &+ (N-1+ 2 c_{k} +
    \lambda (\lambda-N))\int^{\infty}_{0}r^{N-3}|v_k^{\prime}|^{2}dr\nonumber\\
    & + c_k(c_k +(N-4)(2-\lambda))\int^{\infty}_{0}r^{N-5}|v_k^{\prime}|^2
 dr)  \Bigg \} \label{oper}
\end{align}
We split \eqref{oper} as  $\|L_\lambda v\|_{L^2(\rr^N)}^2:=I_{1} + I_{2}+I_{3}$, where 
$$ I_{1} := \int^{\infty}_{0} r^{N-1}|v_k^{\prime\prime}|^{2}dr + (N-1)\int^{\infty}_{0}r^{N-3}|v_k^{\prime}|^{2}dr, $$
$$ I_{2} := \sum^{\infty}_{k=0}\left(2c_{k}\int_{0}^{\infty}r^{N-3} |v_k^{\prime}|^{2}dr + \left(c_{k}^{2}+ 2c_{k}(N-4)\right) \int_{0}^{\infty}r^{N-5}v_k^{2}\right), $$
\begin{align*}
    I_{3}: &= \lambda(\lambda-N)\sum^{\infty}_{k=0}\int^{\infty}_{0}r^{N-3}|v_k^{\prime}|^{2}dr + \lambda(4-N)\sum^{\infty}_{k=0}c_k \int^{\infty}_{0}r^{N-5}|v_k|^{2}dr. 
\end{align*} 
Let us recall the 1-d Hardy inequalities which will be very useful in our approach: 
\begin{align}\label{eq21}
    \int^{\infty}_{0} r^{N-1}|v_k^{\prime\prime}
  |^{2}dr \ge \frac{(N-2)^{2}}{4}\int^{\infty}_{0}r^{N-3}|v_k^{\prime}|^{2}dr 
\end{align}
and
\begin{align}\label{eq22}
  \int^{\infty}_{0}r^{N-3}|v_k^{\prime}|^{2}dr \ge \frac{(N-4)^{2}}{4}\int^{\infty}_{0}r^{N-5}v_k^{2} dr.    
\end{align}

Inequality \eqref{eq21} leads to  
\begin{align}\label{eq23}
    I_{1} \ge \frac{N^{2}}{4} \sum^{\infty}_{k=0}\int^{\infty}_{0}r^{N-3}|v_k^{\prime}|^{2} dr 
\end{align}

As in \cite{Cazacu2}, inequality \eqref{eq22} implies that  
\begin{align}
    I_{2} &\ge \sum^{\infty}_{k=0}c_{k}\left(\frac{(N-4)^2}{2}+2(N-4)+c_k\right)\int^{\infty}_{0}r^{N-5}v_k^{2}dr\label{equ1} \\
    & \geq \frac{N^{2} -2 N -2}{2} \sum^{\infty}_{k=0}c_{k}\int^{\infty}_{0}r^{N-5}v_k^{2}dr \label{eq24}
\end{align}

Summing up inequalities \eqref{eq23} and \eqref{eq24} to \eqref{oper} we obtain 
\begin{align}
    \|L_\lambda v\|_{L^2(\rr^N)}^2 &\ge  \left(\frac{N}{2}-\lambda\right)^2\sum^{\infty}_{k=0} \int^{\infty}_{0}r^{N-3}|v_k^{\prime}|^{2}dr \nonumber\\
    & + \left(\frac{N^{2}-2N-2}{2} + \lambda (4-N)\right) \sum^{\infty}_{k=0} c_{k}\int^{\infty}_{0}r^{N-5}v_k^{2}dr. \label{ineq1}
\end{align}    
If we restrict to $\lambda's$ satisfying 
\begin{equation}
    \frac{N^{2}-2N-2}{2} + \lambda (4-N) \geq 0 \iff 
\left\{\begin{array}{ccc}
       N^{2}-2N-2-2\lambda(N-4)\geq 0,  & N\neq4, \\[5pt]
\lambda \in \rr, & N=4 \\[5pt]
    \end{array}\right. \label{cond}
\end{equation}
from \eqref{ineq1} we obviously reach to 
\begin{align}
    \|L_\lambda v\|_{L^2(\rr^N)}^2 &\ge  \left(\frac{N}{2}-\lambda\right)^2\sum^{\infty}_{k=0} \int^{\infty}_{0}r^{N-3}|v_k^{\prime}|^{2}dr \label{ineq3}
\end{align}  
which in view of \eqref{form3} is equivalent to 
\begin{align}
    \|L_\lambda v\|_{L^2(\rr^N)}^2 &\ge  \left(\frac{N}{2}-\lambda\right)^2\int_{\rr^N} \frac{|x \cdot \nabla v|^2}{|x|^4} dx, \quad \forall (\lambda, N) \textrm{ as in } \eqref{cond}   \label{ineq3b}
\end{align}  
On the other hand,  combining \eqref{oper}, \eqref{eq23}, \eqref{equ1},  the expression of $I_3$  and \eqref{equ1} we simply obtain   
\begin{align}\label{122}
    \|L_\lambda v\|_{L^2(\rr^N)}^2  & \geq \sum_{k=0}^\infty \Bigg \{ \left(\frac{N}{2}-\lambda\right)^2 \int^{\infty}_{0}r^{N-3}|v_k^{\prime}|^{2}dr \nonumber\\ & + c_{k}\left(\frac{(N-4)^2}{2}+2(N-4)+c_k + \lambda (4-N)\right) \int^{\infty}_{0}r^{N-5}v_k^{2}dr\Bigg \}\nonumber\\
    & :=\sum_{k=0}^\infty A_k^N. 
\end{align}
\paragraph{The (trivial) case $N=4$.} From \eqref{ineq3} we obviously get  
\begin{align}
    \|L_\lambda v\|_{L^2(\rr^4)}^2 &\ge \left(2-\lambda\right)^2  \int_{\rr^4} \frac{|x \cdot \nabla v|^2}{|x|^4} dx, \quad \forall \lambda\in \rr,  \label{ineq2}
\end{align}   
which leads to the conclusion \eqref{HR_ineq1}-\eqref{sharp_const}.  
\paragraph{The case $N\neq 4$.} We will focus on the cases 
\begin{equation}\label{condition1}
    N^{2}-2N-2-2\lambda(N-4)<0
\end{equation} since the other cases have been already considered in \eqref{cond}. 
Let us denote by $R(k, N)$ the coefficient of zero order term in \eqref{ineq2} which we rewrite conveniently as   \begin{align}
R(k, N)&:=c_{k}\left(\frac{(N-4)^2}{2}+2(N-4)+c_k + \lambda (4-N)\right) \label{expr1}
\\
&=c_k \left[c_k + (N-4)\left(\frac{N}{2}-\lambda\right)\right]\label{expr2}\\
&=\left[c_k + \frac{(N-4)}{2}\left(\frac{N}{2}-\lambda\right)\right]^2 - \left(\frac{N-4}{2}\right)^2\left(\frac{N}{2}-\lambda\right)^2,\label{expression}
\end{align}
where $c_k=k(N+k-2)$. Notice that $R(0, N)=0$ for any $N\geq 3$. On the other hand in view of \eqref{expr1} and \eqref{condition1} we have that  
$$R(1, N)=(N-1)\left[\frac{N^2-2N-2}{2}+\lambda (4-N)\right]<0.$$
Recall the expression of $A_k^N$: 
\begin{equation}
   A_k^N= \left(\frac{N}{2}-\lambda\right)^2 \int^{\infty}_{0}|v_k^{\prime}|^{2}r^{N-3}dr + R(k, N)\int^{\infty}_{0}r^{N-5}v_k^{2}dr, \quad \forall k\geq 0. \label{A_k^N}
\end{equation}

In view of $R(1, N)<0$ and \eqref{expr2} we are allowed to define the positive integer  
\begin{equation}
    k_0(N, \lambda):=\max \left \{k\geq 1 \ \Big |\  c_k + (N-4)\left(\frac{N}{2}-\lambda\right)< 0\right \}. \label{k_00}
\end{equation}
The natural number $k_0(N, \lambda)$ is uniquely determined and is characterized by the algebraic system 
\begin{equation}
    \left\{ \begin{array}{ll}
    k^2 +k(N-2) + (N-4)\left(\frac{N}{2}-\lambda\right) <0      &  \\
         (k+1)^2 +(k+1)(N-2)+(N-4)\left(\frac{N}{2}-\lambda\right)   \geq 0   & 
    \end{array}\right., k\in \nn\setminus \{0\}.  \label{system2}
\end{equation}
Notice that the discriminant of both quadratic inequalities appearing in \eqref{system2} is 
$$\delta:=4\lambda(N-4)-N^2+4N +4, \quad  \textrm{with }\  \delta > N^2.$$
Each one of the associated quadratic equations in \eqref{k_00} has one positive root given by $$k_{1}^1=\frac{2-N+\sqrt{4\lambda(N-4)-N^2+4N +4}}{2}  \ \textrm{ and }\   k_{1}^2=k_1^1-1, $$ respectively. Then it is easy to see that $k_0(N, \lambda)$ is explicitly given by 
\begin{equation}
    k_0(N, \lambda)=\left\{\begin{array}{cc}
     \frac{-N+\sqrt{4\lambda(N-4)-N^2+4N +4}}{2},    &  \textrm{ if } \frac{-N+\sqrt{4\lambda(N-4)-N^2+4N +4}}{2}\in \nn \\
\left[\frac{-N+\sqrt{4\lambda(N-4)-N^2+4N +4}}{2}
\right]+1,         & \textrm{otherwise}.
    \end{array}\right. \label{k_0}
\end{equation} 
Due to \eqref{expr2} and \eqref{k_00} we have that 
\begin{equation}\label{R(k, N)}
     \left \{ \begin{array}{cc}
       R(k, N) \geq 0, & \forall\  k \geq k_0(N, \lambda)\\
              R(k, N) < 0, & \forall\  k\in \{1,\ldots,k_0(N, \lambda)\}.
    \end{array}\right.
\end{equation}
If $k=0$ then 
\begin{equation}
    A_0^N= \left(\frac{N}{2}-\lambda\right)^2 \int^{\infty}_{0}r^{N-3}|v_k^{\prime}|^{2}dr. \label{0_estimate}
\end{equation}
If $k\geq 1$ taking into account \eqref{expression}, \eqref{A_k^N} and \eqref{R(k, N)} we have 
\begin{align}
    A_k^N &\geq \left(\frac{N}{2}-\lambda\right)^2 \int^{\infty}_{0}r^{N-3}|v_k^{\prime}|^{2}dr + \min_{k=1, k_0(N, \lambda)} \{R(k, N)\} \int^{\infty}_{0}r^{N-5}v_k^{2}dr\nonumber\\
    & = \left(\frac{N}{2}-\lambda\right)^2 \int^{\infty}_{0}r^{N-3}|v_k^{\prime}|^{2}dr \nonumber\\
    &+ \left\{\min_{k=1, k_0(N, \lambda)} \left[c_k+ \frac{N-4}{2}\left(\frac{N}{2}-\lambda\right)\right]^2  - \left(\frac{N-4}{2}\right)^2\left(\frac{N}{2}-\lambda\right)^2\right\} \int^{\infty}_{0}r^{N-5}v_k^{2}dr
\end{align}
Let us consider $\eps>0$ small enough that will be chosen later.  Applying the Hardy inequality \eqref{eq22} we get 
\begin{align}
   A_k^N \geq   &  \left[\left(\frac{N}{2}-\lambda\right)^2-\eps\right] \int^{\infty}_{0}r^{N-3}|v_k^{\prime}|^{2}dr \nonumber\\
   & + \left\{\left(\frac{N-4}{2}\right)^2\left[\eps-\left(
    \frac{N}{2}-\lambda\right)^2\right]+\min_{k=1, k_0(N, \lambda)} \left[c_k +\frac{N-4}{2}\left(\frac{N}{2}-\lambda\right)\right]^2  \right\} \int^{\infty}_{0}r^{N-5}v_k^{2}dr. 
\end{align}
Now choosing $\eps$ such that  
\begin{multline}
    \left(\frac{N-4}{2}\right)^2\left[\eps-\left(
    \frac{N}{2}-\lambda\right)^2\right]+\min_{k=1, k_0(N, \lambda)} \left[c_k +\frac{N-4}{2}\left(\frac{N}{2}-\lambda\right)\right]^2=0 \\
    \iff \eps:=-\left(\frac{2}{N-4}\right)^2\min_{k=1, k_0(N, \lambda)} \left[c_k +\frac{N-4}{2}\left(\frac{N}{2}-\lambda\right)\right]^2+\left(
    \frac{N}{2}-\lambda\right)^2\\
    =-\left(\frac{2}{N-4}\right)^2\min_{k=1, k_0(N, \lambda)} R(k, N) >0
\end{multline} 
we finally get 
\begin{equation}
    A_k^N \geq \left(\frac{2}{N-4}\right)^2\min_{k=1, k_0(N, \lambda)} \left[c_k +\frac{N-4}{2}\left(\frac{N}{2}-\lambda\right)\right]^2 \int^{\infty}_{0}r^{N-3}|v_k^{\prime}|^{2}dr. \label{k_estimates1}
\end{equation}
Formulas \eqref{122}, \eqref{0_estimate}, \eqref{k_estimates1} lead to the conclusion \eqref{sharp_const} of the theorem.

Using  \eqref{prop 1.1} we obtain the inequality \eqref{HR_ineq1}  for any $\lambda \in \rr\setminus\left\{\lambda_n:=\frac{(2n+N-2)^2-4} {2(N-4)} \  | n \in \nn^\star \right\} \setminus \{\frac{N}{2}\}$ and any $N\neq 4$. 
\end{proof}
    \begin{proof}[Proof of Proposition \eqref{prop 1.1}] Let $n$ such that $n^2 +n(N-2)+ \frac{N-4}{2}\left(\frac{N}{2}-\lambda\right)=0$ . Since $n\in \nn $ and $\lambda > \frac{N^2-2N-2}{2(N-4)}$ we get $n=\frac{-(N-2)+\sqrt{4+2(N-4)\lambda}}{2}$ . The condition $ n \ge 1$ is true when $ \lambda \geq  \frac{N^2-4}{2(N-4)}$. Therefore, for $ \lambda < \frac{N^2-4}{2(N-4)}$  there is no $ n$ such that the equation $ n^2 +n(N-2)+ \frac{N-4}{2}\left(\frac{N}{2}-\lambda\right)=0$ holds. It remains to analyze the situation when $ \lambda >\frac{N^2-4}{2(N-4)}$. 
    
    Then we obtain  $\lambda_n:=\frac{(2n+N-2)^2-4}{2(N-4)}$. 

    {\bf The case $k_0=k_0(N, \lambda_n)=\frac{-N+\sqrt{4\lambda_n(N-4)-N^2+4N +4}}{2}$.} Then we get the system 
    \begin{equation}
    \left\{\begin{array}{cc}
       2(2n+N-2)^2 -(N-2)^2= (2k_0 +N)^2,   &  \\
       1\leq n\leq k_0,  & \textrm{ with } k_0, n\in \nn. 
    \end{array}\right. \label{syst1}   
    \end{equation}

 We want to give a complete algebraic description of the pairs $ (n,k_{0})$. Hence we can rewrite the equation such that 
\begin{align*}
    (2k_{0} +N )^{2} + (N-2)^{2} = 2 (2n+N-2)^{2}
\end{align*}
 We denote $ u:= 2k_{0}+N$, $ v:= N-2$ and $ a := 2n+N-2$ such that the previous equation becomes $$ u^{2} + v^{2} = 2 a^{2}.$$
 We can see that $$ (u + v)^{2} + (u-v)^{2} = 2(u^{2} + v^{2}) =4 a^{2}.$$
 Making the change $ p = \frac{u+v}{2}$ and $ q = \frac{u-v}{2}$ we reduce the above equation to a pithagorean equation 
 \begin{align}
     p^{2} + q^{2} = a^{2} 
 \end{align}
 
A known pithagorean triplet is $(p,q,a) =(4m,3m,5m)$, where $ m \in \mathbb{N}$.

We put $ m = v = N-2$ and we have $$ p = 4v, \quad q=3v,\quad a =5v. $$

By definition we have 
$$ u = p+ q = 7v = 2 k_{0} + N \quad \mathrm{and} \quad a = 5v= 2n+N-2.$$ 
Therefore we get $$ k_{0} = \frac{7v-N}{2}= \frac{7(N-2)-N}{2} = 3N-7 \quad \mathrm{and}  \quad n = \frac{a-N +2}{2} =\frac{5(N-2)-N+2}{2} = 2N-4$$

We observe that for $ N \ge 3$ we have $$ n =2N-4 \leq k_{0}= 3N-7$$
Hence the pair $(2N-4, 3N-7)$ is a solution to \eqref{syst1}

We want to prove that the system \eqref{syst1} has a countable infinity of solutions. 
We denote $ x:= 2k_{0}+ N$  and $ y = 2n +N-2$ and after these changes the equation $$ 2(2n+N-2)^{2} -(N-2)^{2}=(2k_{0} +N)^{2}$$ becomes
\begin{align}\label{103}
     x^{2} -2y^{2} = -(N-2)^{2},
\end{align}
which is known as Pell equation. 

Since we have a particular solution $ (n,k_{0})=(2N-4, 3N-7)$ to \eqref{syst1} we get $$ (x_{0}, y_{0}) = (7(N-2), 5(N-2)),$$ a particular solution to \eqref{103}. 

Therefore we define 
\begin{align}\label{106}
    x_{m}+ y_{m}\sqrt{2}= (x_{0}+y_{0}\sqrt{2}) (3+ 2\sqrt{2})^{m}.
\end{align}
It is easily to check that  $$ x_{m}^{2} - 2y_{m}^{2} = (x_{0}^{2}-2y_{0}^{2}) \cdot 1 = -(N-2)^{2}.$$

By  \eqref{106} we deduce that $$ x_{m+1} + y_{m+1}\sqrt{2} = (x_{m}+ y_{m}\sqrt{2})(3+2\sqrt{2}).$$

Therefore we get the second-order linear recurrence relations
 \begin{equation}
    \left\{\begin{array}{cc}
       x_{m+1}=3x_{m}+4y_{m},   &  \\
       y_{m+1}=2x_{m}+3y_{m} & . 
    \end{array}\right. \label{107}   
    \end{equation}
The system \eqref{107} yields that $ x_{m+1} \equiv x_{m} \pmod{2}$ and $ y_{m+1} \equiv y_{m} \pmod{2}$. On the other hand we observe that $ x_{0} \equiv N \pmod{2}$ and $ y_{0} \equiv N \pmod{2}$ and putting all together we get by induction that $$ x_{m} \equiv N \pmod{2}\quad  \mathrm{and}\quad y_{m} \equiv N \pmod{2}.$$

Hence we obtain the integer numbers $$ k_{0_{m}} = \frac{x_{m}-N}{2}\in \mathbb{Z} \quad \mathrm{and} \quad n_{m} =\frac{y_{m}-N+2}{2}\in \mathbb{Z} $$

By conjugation the equality \eqref{106} we obtain 
$$ x_{m}= \frac{N-2}{2}((3+ 2\sqrt{2})^{m} + (3-2\sqrt{2})^{m}) \quad \mathrm{and} \quad y_{m}= \frac{N-2}{2\sqrt{2}}((3+ 2\sqrt{2})^{m} - (3-2\sqrt{2})^{m}) $$

For $ N \geq 3$ we note that $ x_{0}\geq y_{0}$ and keeping into account that $ x_{m}, y_{m}>0$ and the system \eqref{107} holds then we obtain  $$ x_{m+1}-y_{m+1} =x_{m}+y_{m}>0.$$

By induction we get $ x_{m}\geq y_{m}$ and therefore we have $ k_{0_{m}}\geq n_{m}$.  

Since $ y_{m+1}= 2x_{m}+ 3y_{m}  \geq 5 y_{m} \geq y_{m}$ we deduce that $ y_{m}\geq y_{0}\geq N$ and therefore we have $$ n_{m}= \frac{y_{m}-N+2}{2} \geq 1.$$

{\bf The case $k_0=k_{0}(N, \lambda_{n})=\left[\frac{2-N+\sqrt{4\lambda_{n}(N-4)-N^2+4N +4}}{2}\right] = \left[\frac{2-N +\sqrt{2(2n+N-2)^{2}-(N-2)^{2}}}{2}\right]$. 
} It is sufficient to prove that $ n \leq k_{0}$.

By a direct computation we deduce that $$ n \leq\frac{2-N +\sqrt{2(2n+N-2)^{2}-(N-2)^{2}}}{2} $$ since $ (N-2)^{2}\leq (2n+N-2)^{2}$ for $ N \geq 3$. 

Therefore $ n \leq \left[\frac{2-N +\sqrt{2(2n+N-2)^{2}-(N-2)^{2}}}{2}\right]=k_{0}$. 

Conversely, in the both cases if $\min_{k=1, k_0(N, \lambda)} \left[c_k +\frac{N-4}{2}\left(\frac{N}{2}-\lambda \right)\right]^{2}=0$ we get that there exists $ n \leq k_{0}(N,\lambda)$ such that $\lambda_{n}= \frac{(2n+N-2)^2-4}{2(N-4)}$.
\end{proof}


\begin{proof}[Proof of Theorem \eqref{Ineq_lambda}]
We first have 
\begin{align}
    \int_{\rr^N}  \left | -
\Delta v + \lambda\frac{ v}{|x|^2} 
\right|^2  dx
     =\int_{\mathbb{R}^{N}}|\Delta v |^{2}dx + \lambda^{2}\int_{\mathbb{R}^{N}}\frac{ v^{2} }{|x|^{4}}dx + 2 \lambda \int_{\mathbb{R}^{N}}- \Delta{v }\frac{ v }{|x|^{2}} dx.  \label{inequality}
\end{align}

Next we compute separately the mixt term $ I := \int_{\mathbb{R}^{n}}-\Delta v \frac{ v }{|x|^{2}}dx $ in \eqref{mixt}. Using again Einstein's summation convection we obtain 
\begin{align}\label{calcul1}
   I = &\int_{\mathbb{R}^{N}}-\Delta v \frac{v}{|x|^{2}}dx = - \int_{\mathbb{R}^{N}}v_{x_{i}x_{i}}\frac{v}{|x|^{2}} = \int_{\mathbb{R}^{N}}v_{x_{i}}\frac{v_{x_{i}}|x|^{2}-2vx_{i}}{|x|^{4}}  \nonumber \\
   & = \int_{\mathbb{R}^{N}} v_{x_{i}}\left(\frac{v_{x_{i}}}{|x|^{2}} - \frac{2vx_{i}}{|x|^{4}}\right) = \int_{\mathbb{R}^{N}}\frac{|\nabla v|^{2}}{|x|^{2}} - \int_{\mathbb{R}^{N}}\nabla (v^{2})\frac{x_{i}}{|x|^{4}}  \nonumber \\
   & = \int_{\mathbb{R}^{N}} \frac{|\nabla v|^{2}}{|x|^{2}} + \int_{\mathbb{R}^{N}} v^{2}\frac{\partial}{\partial x_{i}}\left(\frac{x_{i}}{|x|^{4}}\right) dx  \nonumber \\
  &=  \int_{\mathbb{R}^{N}}\frac{|\nabla v|^{2}}{|x|^{2}}dx + N \int_{\mathbb{R}^{N}}v^{2}\frac{|x|^{4}}{|x|^{8}} - \int_{\mathbb{R}^{N}}4\frac{v^{2}|x|^{2}x_{i}^{2}}{|x|^{8}} dx \nonumber \\
  &= \int_{\mathbb{R}^{N}} \frac{|\nabla v|^{2}}{|x|^{2}} + N \int_{\mathbb{R}^{N}}\frac{v^{2}}{|x|^{4}} - 4 \int_{\mathbb{R}^{N}}\frac{v^{2}}{|x|^{4}}\nonumber \\
  &= \int_{\mathbb{R}^{N}}\frac{|\nabla v|^{2}}{|x|^{2}}dx + (N-4)\int_{\mathbb{R}^{N}}\frac{v^{2}}{|x|^{4}}dx.  
\end{align}

Coupling \eqref{inequality} and \eqref{calcul1} we summarize  
\begin{align*} 
     \int_{\mathbb{R}^{N}} \left|-\Delta v + \lambda \frac{v}{|x|^{2}}\right|^{2}dx = \int_{\mathbb{R}^{N}}|\Delta v|^{2}dx + (\lambda ^{2} + 2\lambda (N-4)) \int_{\mathbb{R}^{N}}\frac{v^{2}}{|x|^{4}}dx + 2 \lambda \int_{\mathbb{R}^{N}}\frac{|\nabla v|^{2}}{|x|^{2}}. 
\end{align*}
It is obvious that for $ \lambda > 0$ we have that $$ \int_{\mathbb{R}^{N}}\left|-\Delta v + \lambda \frac{v}{|x|^{2}}\right|^2 dx  \ge (\lambda ^{2} + 2\lambda (N-4))\int_{\mathbb{R}^{N}} \frac{v^{2}}{|x|^{4}}dx .$$

We apply again the spherical harmonics decomposition. Identities \eqref{eq20}, \eqref{form2}, \eqref{form3} 
and 
\begin{equation}
\int_{\mathbb{R}^{N}}\frac{v^{2}}{|x|^{4}}= \sum_{k=0}^{\infty}\int_{0}^{\infty}r^{N-5}v_{k}^{2}dr  \label{form11}
\end{equation}
 lead to 
\begin{align}\label{form14}
    \int_{\mathbb{R}^{N}}\left|-\Delta v + \lambda\frac{v}{|x|^{2}}\right|^{2} &=  \sum_{k=0}^{\infty} \Bigg \{ \int^{\infty}_{0} r^{N-1}|v_k^{\prime\prime}|^{2}dr +(N-1+2\lambda + 2c_{k})\int_{0}^{\infty}r^{N-3}|v_{k}^{\prime}|^{2} \nonumber\\
    &+(\lambda^{2} + 2\lambda(N-4))\int^{\infty}_{0}r^{N-5}|v_k|^{2}dr\nonumber\\
    & + c_k(c_k +2(N-4)+ 2\lambda)\int^{\infty}_{0}r^{N-5}|v_k|^2
 dr)  \Bigg \} 
\end{align}
We split $ \int_{\mathbb{R}^{N}} \left|-\Delta v +\lambda \frac{v}{|x|^{2}}\right|^{2} dx = I_{1} + I_{2} + I_{4}$, where $I_1$ and $I_2$ are as in the proof of Thm. \ref{Teorema 1}, namely 
$$ I_{1} = \sum_{k=0}^{\infty}\left(\int^{\infty}_{0} r^{N-1}|v_k^{\prime\prime}|^{2}dr + (N-1)\int^{\infty}_{0}r^{N-3}|v_k^{\prime}|^{2}dr\right), $$
$$ I_{2} = \sum^{\infty}_{k=0}\left(2c_{k}\int_{0}^{\infty}r^{N-3} |v_k^{\prime}|^{2}dr + \left(c_{k}^{2}+ 2c_{k}(N-4)\right) \int_{0}^{\infty}r^{N-5}v_k^{2}\right), $$
whereas 
\begin{align*}
    I_{4}: &=  \sum^{\infty}_{k=0}\left[2\lambda\int^{\infty}_{0}r^{N-3}|v_k^{\prime}|^{2}dr +(\lambda^{2} + 2 \lambda (N-4) + 2c_{k}\lambda)\int^{\infty}_{0}r^{N-5}|v_k|^{2}dr\right]. 
\end{align*} 
In view of the 1-d Hardy inequalities  \eqref{eq21}-\eqref{eq22} combined with the lower bound for $I_1$ and $I_2$ obtained in \eqref{eq23}-\eqref{eq24}  we get that 

\begin{align}
    \int_{\mathbb{R}^{N}}|-\Delta v+ \lambda \frac{v}{|x|^{2}}|^{2}dx &\ge \left(\frac{N^{2}}{4} + 2\lambda\right)\sum_{k=0}^{\infty}\int_{0}^{\infty}r^{N-3}|v_{k}^{\prime}|^{2}   \\ \nonumber 
    & + \left(\frac{N^{2} -2N-2}{2} +  2\lambda\right)\sum_{k=0}^{\infty}c_{k}\int_{0}^{\infty}r^{N-5}v_{k}^{2}dr \\ \nonumber
    &+ (\lambda^{2} + 2 \lambda(N-4))\sum_{k=0}^{\infty}\int_{0}^{\infty}r^{N-5}|v_{k}|^{2}dr \\ \nonumber 
    &\geq \left[\left(\frac{N^{2}}{4}+2\lambda\right)\left(\frac{N-4}{2}\right)^{2} + \lambda^{2} + 2\lambda(N-4)\right]\sum_{k=0}^{\infty}\int_{0}^{\infty}r^{N-5}|v_{k}|^{2}dr + \\ \nonumber 
    & + \left(\frac{N^{2}-2N-2}{2} + 2\lambda\right) \sum_{k=0}^{\infty}c_{k}\int_{0}^{\infty}r^{N-5}v_{k}^{2}dr 
\end{align}

A simple computation shows that for dimension $ N \geq 1$ and $ \lambda \in \mathbb{R}\setminus\{-\frac{N(N-4)}{4}\}$ we have $$ \left(\frac{N^{2}}{4}+2\lambda\right)\left(\frac{N-4}{2}\right)^{2} + \lambda^{2} + 2\lambda(N-4) =\left(\lambda + \frac{N(N-4)}{4}\right)^{2}>0 $$

For $ \lambda > \frac{-N^{2} +2N +2}{4} $ we get $ \frac{N^{2} -2N-2}{2} +2\lambda > 0 $. 

Hence for dimension $ N \geq 5$ and $ \lambda \in (\frac{-N^{2} + 2N +2}{4}, \infty) \setminus\{-\frac{N(N-4)}{4}\}$
we obtain that
$$ \int_{\mathbb{R}^{N}}|-\Delta v +\lambda \frac{v}{|x|^{2}}|^{2}dx \geq \left[\left(\frac{N^{2}}{4} +2\lambda\right)\left(\frac{N-4}{2}\right)^{2} + \lambda^{2} + 2 \lambda(N-4)\right]\int_{\mathbb{R}^{N}}\frac{v^{2}}{|x|^{4}}dx.$$
\end{proof}

\section{Examples and counterexamples}\label{section3}

Let us consider problem \eqref{p1}  in the unit ball

\begin{subequations}\label{eq109}
\begin{empheq}[left=\empheqlbrace]{align}
-\Delta v +\lambda\frac{x\cdot\nabla v}{|x|^{2}}
&=|x|^{\alpha},
\text{in }B_{1}(0),\label{eq109a}\\
v&=0,
\text{on }\partial B_{1}(0),\label{eq109b}
\end{empheq}
\end{subequations}

with the source $ f(x)=|x|^{\alpha} \in L^{2}(B(0,1))$ when $\alpha > -\frac{N}{2}$.

The following proposition proves the assertion (b) from theorem \eqref{Teorema 1}. 

\begin{prop}\label{ Teorema 4}
    Let $ \alpha \in (-\frac{N}{2},\infty) \setminus \{-2,\lambda-N\}$.  Radial solutions of  problem \eqref{eq109a} are of the form 
    \begin{align}
        v(x) = - \frac{|x|^{\alpha+2}}{(\alpha+2)(\alpha+N-\lambda)} 
+ C_1 \, |x|^{-N+2+\lambda} + C_2, \quad C_{1}, C_{2} \in \mathbb{R}.
    \end{align}
    Moreover, if $ \lambda \in (\frac{N-2}{2},\frac{N}{2}]$ thus $ v \in H_{0}^{1}(B_{1}(0)) \setminus H^{2}(B_{1}(0))$. 
\end{prop}
\begin{proof}[Proof of Proposition \eqref{ Teorema 4}]

We look for a radial solution of the form $v(x) = g(|x|) = g(r)$.

Thus, the equation \eqref{eq109a} becomes:

$$- g''(r) - \frac{N-1}{r} g'(r) + \lambda \frac{g'(r)}{r} r = r^\alpha.$$
We set $g'(r) = w(r)$ and we get 
\begin{align}\label{120}
    - w'(r) - \frac{w(r)}{r}(N-1-\lambda) = r^\alpha.
\end{align}

Using the method of variation of constants we look for solutions $ w$ of the form 
$$\overline{w}(r) = C (r)\cdot r^{-(N-1-\lambda)}.$$

By substitution in \eqref{120} we get using the integration that there exists a constant $ C_{1} \in \mathbb{R}$ such that

$$ C(r)=-\frac{r^{\alpha+N-\lambda}}{\alpha + N-\lambda}+C_{1} \quad \mathrm{and}\quad w(r)= -\frac{r^{\alpha+1}}{\alpha + N-\lambda} +C_{1}r^{-N+1+\lambda}$$
 
Integrating we obtained that 
$$ v(x) = - \frac{|x|^{\alpha+2}}{(\alpha+2)(\alpha+N-\lambda)} 
+ C_1 \, |x|^{-N+2+\lambda} + C_2, $$ 
where $ \alpha \notin \{-2, \lambda -N\}$ and $ C_{2} \in \mathbb{R}$.  

 If $ C_{1} + C_{2} = -\frac{1}{(\alpha+2)(\alpha +N -\lambda)}$ we have the boundary condition \eqref{eq109b} satisfied. 
 
We want to find the parameters $ \lambda$ such that $v \in H_{0}^1(B(0,1))$.

This is equivalent to
$$ \int_{B_{1}(0)}(|v(x)|^{2} + |\nabla v(x)|^{2})dx< \infty.$$

For $ \alpha > -\frac{N}{2}$ we get that 
\begin{align*}
    \left\{\begin{array}{ll} 
 \int_{B_{1}(0)} ||x|^{\alpha+2}|^{2}dx = \int_{0}^{1} r^{2(\alpha+2) +N-1}dr < \infty &		\\
 \int_{B_{1}(0)} |\nabla |x|^{\alpha+2}|^{2} dx =\int_{0}^{1} r^{2(\alpha+1)+ N-1}dr < \infty & \\ 
 \int_{B_{1}(0)}|D^{2}|x|^{\alpha+2}|^{2}dx = \int_{0}^{1}r^{2\alpha + N-1}dr < \infty,
	\end{array}\right.
\end{align*}

which means that $ |x|^{\alpha + 2} \in H^{2}(B_{1}(0))$. 

On the other hand, for $ \lambda > \frac{N-2}{2}$ we observe that 
\begin{align*}
    \left\{\begin{array}{ll} 
 \int_{B_{1}(0)} ||x|^{-N+2+\lambda}|^{2}dx = \int_{0}^{1} r^{2(-N+2+\lambda) +N-1}dr < \infty &		\\
 \int_{B_{1}(0)} |\nabla |x|^{-N+2+\lambda}|^{2}dx = \int_{0}^{1}r^{2(-N+1+\lambda)+N-1}dr <\infty,\\ 
	\end{array}\right.
\end{align*}
which implies that $ |x|^{-N + 2+\lambda} \in H^{1}(B(0,1))$.

Putting all together for $ \alpha \in (-\frac{N}{2},\infty) \setminus \{-2,\lambda-N\} ,$ $ \lambda > \frac{N-2}{2}$ and 
$ C_{1} + C_{2} = -\frac{1}{(\alpha+2)(\alpha +N -\lambda)}$
we deduce that $ v \in H_{0}^{1}(B_{1}(0))$. 

For $ \lambda >\frac{N}{2}$ we obtain that $$\int_{B_{1}(0)}|D^{2}|x|^{-N+ 2 +\lambda }|^{2}dx= \int_{0}^{1}r^{2(-N+ \lambda)+ N-1}dr < \infty$$
which gives us that 
$|x|^{-N+ 2 +\lambda } \in H^{2}(B_{1}(0))$. 

Hence we showed that $v \in H_{0}^{1}(B_{1}(0))\setminus H^{2}(B_{1}(0))$ for $ \lambda \in (\frac{N-2}{2},\frac{N}{2}]$.  
\end{proof}

\begin{obs}

   For $ N=3$ and $ \lambda \in (\frac{1}{2},\frac{3}{2}]$ we get an example of solution $ v \in H_{0}^{1}(B(0,1))\setminus H^{2}(B(0,1))$ to problem  \eqref{ Teorema 4}, where the source $ |x|^{\alpha} \in L^{2}(B(0,1))$. 

   A similar counterexample to the Calderon-Zygmund theory we obtained for $ \lambda \in (1,2]$ and $ N=4$.
\end{obs}

Now we focus on the following homogeneous singular elliptic equation 

\begin{align}\label{eq113}
-\Delta v+ \lambda \frac{x\nabla v}{|x|^{2}}=0,
\end{align}
where $ \lambda < \frac{N-2}{2}$. 

\begin{prop}\label{ Teorema 5}
    Let $ \beta \in \left\{\frac{N-\lambda + \sqrt{(N-\lambda)^{2}+4\lambda}}{2},\frac{N-\lambda -\sqrt{(N-\lambda)^{2}+4\lambda}}{2}\right\}$.
    Then $ v(x) = \frac{x_{N}}{|x|^{\beta}}$ is a solution to \eqref{eq113}.
    If $ N \geq 5$ and $ \lambda \in \left[\frac{N^{2}}{2(N-2)},\frac{N^{2}-4}{2(N-4)}\right]$ we have that $ v \in H^{1}(B_{1}(0))\setminus H^{2}(B_{1}(0))$.
    
    In addition, if $ N=4$ and $ \lambda \in[4,\infty)$ or $ N = 3$ and $ \lambda \in \left(-\infty,-\frac{5}{2}\right] \cup \left[\frac{9}{2}, \infty\right)$ we have that $ v\in H^{1}(B_{1}(0))\setminus H^{2}(B_{1}(0))$. 
\end{prop}
\begin{proof}[Proof of Proposition \eqref{ Teorema 5}]

We substitute $ v(x)= x_{N}|x|^{-\beta}$ in the equation \eqref{eq113} and computing we have 
$$ \nabla v = |x|^{-\beta} e_{N} - \beta x_{N} |x|^{-\beta-2}x \quad \textrm{and}\quad x\nabla v(x)=(1-\beta) x_{N}|x|^{-\beta}.$$ 
On the other hand, we get 
$$ \frac{x\nabla v}{|x|^{2}}= (1-\beta)\frac{x_{N}}{|x|^{\beta +2}} \quad \textrm{and}\quad \Delta(x_{N}|x|^{-\beta})= -\beta(N-\beta)\frac{x_{N}}{|x|^{\beta +2}}. $$
Then the equation becomes $ [\beta(N-\beta) + \lambda(1-\beta)]\frac{x_{N}}{|x|^{\beta +2}}=0$ which implies that $$ \beta^{2}-(N-\lambda)\beta -\lambda = 0.$$
Therefore we obtain  that $\beta \in \left\{\frac{N-\lambda + \sqrt{(N-\lambda)^{2}+4\lambda}}{2},\frac{N-\lambda -\sqrt{(N-\lambda)^{2}+4\lambda}}{2}\right\} $.

Firstly, we choose $ \beta_{-} = \frac{N-\lambda -\sqrt{(N-\lambda)^{2}+4\lambda}}{2}. $

It is important to see that $$ |v(x)|= \left|\frac{x_{N}}{|x|^{\beta}}\right|\leq |x|^{1-\beta},\quad |\nabla v(x)|\leq |x|^{-\beta} \quad \mathrm{and} \quad |\Delta v(x)|\leq |x|^{-\beta-1}.$$

For $ \beta_{-} <\frac{N}{2}$ we observe that $$ \int_{B_{1}(0)}v^{2}dx= \int_{0}^{1}r^{N+1-2\beta_{-}}dr < \infty \quad \mathrm{and}\quad  \int_{B_{1}(0)}|\nabla v|^{2}dx =\int_{0}^{1}r^{N-1-2\beta_{-}}dr < \infty$$ which implies $ v \in H^{1}(B_{1}(0))$. 

We want to find the parameters $ \lambda$ such that $\frac{N-\lambda -\sqrt{(N-\lambda)^{2}+4\lambda}}{2} <\frac{N}{2} $. By a direct computation we get
\begin{align}\label{121}
    -\lambda < \sqrt{(N-\lambda)^{2} + 4\lambda}.
\end{align}
If $ \lambda \geq 0$ it is obvious that the inequality \eqref{121} holds. 

For $ \lambda <0$ the relation \eqref{121} is equivalent with $ \lambda < \frac{N^{2}}{2(N-2)}$,  which is possible. 

Therefore we showed  that for any $ \lambda \in \mathbb{R}$ we have $\beta_{-} < \frac{N}{2} $ which means that $ v \in H^{1}(B_{1}(0))$. 

We want to find $ \lambda$ such that $ v \notin H^{2}(B(0,1))$ and then we impose that $ D^{2}v \notin  L^{2}(B(0,1)) $. This requires that $$\int_{B(0,1)}|D^{2}v(x)|^{2} dx=\int_{0}^{1}r^{N-3-2\beta_{-}}dr = \infty.$$
The previous integral is infinite if $ \beta_{-} \geq \frac{N-2}{2}$ and this is  equivalent with 
\begin{align}\label{125}
    \sqrt{(N-\lambda)^{2} + 4\lambda}\leq 2-\lambda.
\end{align}
If $ \lambda >2$ \eqref{125} is impossible, but if $ \lambda \leq 2$ the relation \eqref{125} becomes $$ 2\lambda(4-N)\leq 4-N^{2}.$$

If $ N=3$ we observe that $ \lambda \in(-\infty, -\frac{5}{2}]$ such that $ v \in H^{1}(B(0,1))\setminus H^{2}(B(0,1))$. 

If $ N=4$ we get $ 0\leq -12$ which is false and if $ N \geq 5$  we obtain $ \lambda \geq \frac{N^{2}-4}{2(N-4)}$, but $ \lambda \leq 2$ and again is impossible. 

Now, we analyze the case when $ \beta_{+}=  \frac{N-\lambda + \sqrt{(N-\lambda)^{2}+4\lambda}}{2}$. By computation the condition $ \beta_{+} <\frac{N}{2}$ works when 
\begin{align}\label{eq114}
     \lambda \geq \frac{N^{2}}{2(N-2)},
\end{align}
which ensures $ v \in B_{1}(0)$. 

The condition $ v \notin H^{2}(B(0,1))$ lead to  $ \beta_{+} \geq \frac{N-2}{2}$ which means that 
\begin{align}\label{126}
    N^{2} -2N\lambda +4\lambda \geq -4\lambda +4.
\end{align}

For $ N = 3$ inequalities \eqref{eq114} and \eqref{126} hold when $ \lambda > \frac{9}{2}$ and in this case $ v \in H_{0}^{1}(B_{1}(0))\setminus H^{2}(B_{1}(0))$. 

If $ N =4$ and $ \lambda > 4$ we get $ v \in H_{0}^{1}(B_{1}(0))\setminus H^{2}(B_{1}(0)) $. 

For $ N \geq 5$ we have 
\begin{align}\label{eq115}
    \lambda \leq \frac{N^{2}-4}{2(N-4)}.
\end{align}

Combining \eqref{eq114} and \eqref{eq115} we can deduce the conlusion. 
\end{proof}

Using the previous proposition we can build an interesting example for assertion (c) from theorem \eqref{Teorema 1}.  

\begin{prop}\label{prop}
    Let $ \beta \in \left\{\frac{N-\lambda + \sqrt{(N-\lambda)^{2}+4\lambda}}{2},\frac{N-\lambda -\sqrt{(N-\lambda)^{2}+4\lambda}}{2}\right\}$ and 
    \begin{align}\label{eq2}
 \xi(x):=  \left\{\begin{array}{ll} 
1,  &		\textrm{ for} \quad  |x|< \frac{1}{4}\\
	0, &  \textrm{ for }  |x|\geq \frac{1}{2},\\ 
	\end{array}\right.
\end{align} 
    Then $v(x) =\xi(x) \frac{x_{N}}{|x|^{\beta}}$ is a solution to \eqref{p1}, where $ f(x)=-2\nabla \xi\nabla \left(\frac{x_{N}}{|x|^{\beta}}\right) - \Delta \xi \frac{x_{N}}{|x|^{\beta}}  + \lambda \frac{x_{N}}{|x|^{\beta}} \frac{ x\nabla{\xi}}{|x|^{2}}.$

Moreover, we have that  
\begin{align}
    \left\{\begin{array}{ll} 
v\in H_{0}^{1}(B_{1}(0))\setminus H^{2}(B_{1}(0)),  &		\textrm{ if} \quad  \lambda \in \left(-\infty, -\frac{5}{2}\right] \cup \left[\frac{9}{2},\infty\right) \quad \textrm{and}\quad N=3\\
v\in  H_{0}^{1}(B_{1}(0))\setminus H^{2}(B_{1}(0))	,&\textrm{ if} \quad  \lambda \in [4,\infty) \quad \textrm{and}\quad N=4\\ 
v\in  H_{0}^{1}(B_{1}(0))\setminus H^{2}(B_{1}(0)),  &		\textrm{ if} \quad  \lambda \in \left[\frac{N^{2}}{2(N-2)},\frac{N^{2}-4}{2(N-4)}\right] \quad \textrm{and}\quad N\geq5\\ 
	\end{array}\right.
\end{align}   
\end{prop}
\begin{proof}[Proof of Proposition \eqref{prop}]
Using the linearity of the operator with drift term and taking into account the proposition \eqref{ Teorema 5} we get that $$ -\Delta v+\lambda \frac{x\nabla v}{|x|^{2}}= -2\nabla \xi\nabla \left(\frac{x_{N}}{|x|^{\beta}}\right) - \Delta \xi \frac{x_{N}}{|x|^{\beta}}  + \lambda \frac{x_{N}}{|x|^{\beta}} \frac{ x\nabla{\xi}}{|x|^{2}}:=f(x)$$
It is obvious that $ f \in L^{2}(B_{1}(0))$ since $ f=0$ in $ B_{\frac{1}{4}}(0) \cup B_{\frac{1}{2}}^{c}(0)$ and f is bounded in $ \frac{1}{4}\leq  |x|\leq \frac{1}{2}$.
 If $ |x|> \frac{1}{4}$ we get $ v \in H^{2}(B_{\frac{1}{2}}(0))$ by standard elliptic regularity.  
 
 If $ |x|<\frac{1}{4}$ we obtain that $ v(x)$ coincides with $\frac{x_{N}}{|x|^{\beta}}$ and in proposition \eqref{ Teorema 5} we have the description of the regularity of  $\frac{x_{N}}{|x|^{\beta}}$. Hence the conclusion follows. 
\end{proof}
Next, our goal is to construct an example which guarantees the assertion (b) of the theorem \eqref{Teorema 2} for the problem \eqref{p2} in the unit ball

\begin{subequations}\label{eq126}
\begin{empheq}[left=\empheqlbrace]{align}
-\Delta v +\lambda\frac{v}{|x|^{2}}
&=|x|^{\alpha},
\text{in }B_{1}(0),\label{eq126a}\\
v&=0,
\text{on }\partial B_{1}(0),\label{eq126b}
\end{empheq}
\end{subequations}

Again, we require that $ \alpha >-\frac{N}{2}$ since we want $ |x|^{\alpha} \in L^{2}(B_{1}(0))$.
Also, we denote $$ \lambda_{\ast}= -\frac{(N-2)^{2}}{4}$$
\begin{prop}\label{ Teorema 8}
Let $ \lambda > \lambda_{\ast} $ and $ \alpha \in (-\frac{N}{2}, \infty)$ with $ \lambda \neq (\alpha +2)(\alpha +N)$.     Radial solutions of problem \eqref{eq126a} are of the form 
    \begin{align}
        v(x) = \frac{|x|^{\alpha +2}}{\lambda - (\alpha+2)(\alpha+N)} 
+ C_1 \, |x|^{-\frac{N-2}{2}+\sqrt{\lambda+ \lambda_{\ast}}} + C_2|x|^{-\frac{N-2}{2}-\sqrt{\lambda + \lambda_{\ast}}}.
    \end{align}
    Moreover, if $ \lambda \in \left(-\frac{(N-2)^{2}}{4},-\frac{N(N-4)}{4}\right]$ thus $ v \in H_{0}^{1}(B_{1}(0)) \setminus H^{2}(B_{1}(0))$. 
\end{prop}
\begin{proof}[Proof of Proposition \eqref{ Teorema 8}]

We look for a radial solution of the form $v(x) = g(|x|) = g(r)$.

Thus, the equation becomes:

$$- g''(r) - \frac{N-1}{r} g'(r) + \lambda \frac{g(r)}{r^{2}}  = r^\alpha.$$

We multiply by $ r^{2}$ the previous identity and we get that 
\begin{align}\label{eq127}
    -r^{2}g^{\prime \prime}(r)-(N-1)rg^{\prime}(r) + \lambda g(r) = r^{\alpha +2}
\end{align}
  Firstly we solve the homogeneous equation of \eqref{eq127} 
$$-r^{2}g^{\prime \prime}(r)-(N-1)rg^{\prime}(r) + \lambda g(r) = 0 $$
and we get the constants $ C_{1}, C_{2} \in \mathbb{R}$ and the homogeneous solutions 
$$g_{h}(r)= C_{1} r^{-\frac{N-2}{2}+\sqrt{\lambda+ \lambda_{\ast}}} + C_2r^{-\frac{N-2}{2}-\sqrt{\lambda + \lambda_{\ast}}}$$
We proceed as follows by means of the  method of
variation of constants and we find a particular solution to the equation \eqref{eq127}
$$ g_{p}(r)=\frac{r^{\alpha +2}}{\lambda - (\alpha+2)(\alpha+N)} $$
Therefore we have that
 \begin{align}
        v(x) = \frac{r^{\alpha +2}}{\lambda - (\alpha+2)(\alpha+N)} 
+ C_1 \, |x|^{-\frac{N-2}{2}+\sqrt{\lambda+ \lambda_{\ast}}} + C_2|x|^{-\frac{N-2}{2}-\sqrt{\lambda + \lambda_{\ast}}},
    \end{align}
where $ \lambda > \lambda_{\ast}$ and $ \lambda \neq (\alpha+2)(\alpha+N)$. 

If $ C_{1}+C_{2} = \frac{1}{(\alpha+2)(\alpha+N)-\lambda}$ we have the boundary condition \eqref{eq126b} verified. 
    
By the proof of the proposition \ref{ Teorema 4} we observe that $ |x|^{\alpha+ 2} \in H^{2}(B_{1}(0))$ for $ \alpha> -\frac{N}{2}$.  

We choose $ C_{2}=0$ and we want to find parameters $ \lambda$ such that $ v \in H^{2}(B_{1}(0))$. 

For $ \lambda > -\frac{(N-2)^{2}}{4}$ we obtain

\begin{align*}
    \left\{\begin{array}{ll} 
 \int_{B_{1}(0)} ||x|^{-\frac{N-2}{2}+\sqrt{\lambda +\lambda_{\ast}}}|^{2}dx = \int_{0}^{1}r^{-N+2 +2 \sqrt{\lambda_{\star}+\lambda} +N-1 }dr < \infty &		\\
 \int_{B_{1}(0)} |\nabla |x|^{-\frac{N-2}{2}+\sqrt{\lambda +\lambda_{\ast}}}|^{2}dx = \int_{0}^{1}r^{-N+2\sqrt{\lambda_{\star}+ \lambda}+N-1}dr <\infty,\\ 
	\end{array}\right.
\end{align*}
which gives us that $ |x|^{-\frac{N-2}{2}+\sqrt{\lambda +\lambda_{\ast}}} \in H^{1}(B_{1}(0))$ and therefore $ v \in H_{0}^{1}(B_{1}(0))$. 

For $ \lambda >-\frac{N(N-4)}{4} $ we have 
$$ \int_{B_{1}(0)}|D^{2}|x|^{-\frac{N-2}{2}+\sqrt{\lambda +\lambda_{\ast}}}|^{2}dx = \int _{0}^{1}r^{-N-2+2\sqrt{\lambda_{\ast}+\lambda}+N-1}dr = \int_{0}^{1}r^{-3 + 2 \sqrt{\lambda_{\ast}+\lambda}}dr < \infty,$$
which means that $ |x|^{-\frac{N-2}{2}+\sqrt{\lambda +\lambda_{\ast}}} \in H^{2}(B_{1}(0))$. 

Hence we proved that  $ v \in H_{0}^{1}(B_{1}(0)) \setminus H^{2}(B_{1}(0))$ for $ \lambda \in \left(-\frac{(N-2)^{2}}{4},-\frac{N(N-4)}{4}\right]$. 

\end{proof}

\section{Proof of the main regularity results}

Firstly, we state the maximum principle for operator $ -\Delta + \lambda \frac{x\nabla \cdot}{|x|^{2}}$, which is an important tool in the proof of the main result. 

\begin{prop}\label{ Teorema 2}
    Let $ \Omega$ denote a smoothly bounded domain of $ \mathbb{R}^{N}$, $ N \ge 3$. Assume that $ v \in H^{1}( \Omega) $  satisfies, in the weak sense,
\begin{align}\label{eq102}
    \left\{\begin{array}{cc} 
-\Delta v +\lambda\frac{x\cdot \nabla v}{|x|^{2}} \geq 0,  &		\textrm{ in} \quad  \Omega \\
	v\ge0, &  \textrm{ on }  \partial \Omega,\\ 
	\end{array}\right.
\end{align}
where $ \lambda < \frac{N-2}{2}$. Then $ v \ge 0$ in $ \Omega$. 
\end{prop}
\begin{proof}[Proof of Proposition  \eqref{ Teorema 2}]

Let $ v$ denote a solution to \eqref{eq102}. We know that 
\begin{align}\label{eq103}
    \int_{\Omega}\left[ \nabla v \cdot \nabla \phi + \lambda \frac{x\nabla v}{|x|^{2}}\phi\right ]dx \geq 0,
\end{align}
for any $ \phi \in H_{0}^{1}(\Omega)$ with $ \phi \geq 0$.

We set 
\begin{align}\label{eq100}
 v^{-}(x)= \left\{\begin{array}{ll} 
-v,  &		v\leq 0\\
	0, &  v\geq 0.\\ 
	\end{array}\right.
\end{align}
It is known that $ v^{-} \in H^{1}(\Omega)$ when $ v \in H^{1}(\Omega)$. Since $ v \geq 0$ on $ \partial \Omega$ we have $ v^{-} = 0$ on $ \partial \Omega$. 
We can set $ \phi:=v^{-} \in H^{1}_{0}(\Omega)$. Keeping into account that $ v = v^{+}-v^{-}$ the inequality \eqref{eq103} becomes 
\begin{align*}
     -\int_{\Omega}|\nabla v^{-}|^{2}dx + \frac{\lambda}{2}(N-2)\int_{\Omega}\frac{(v^{-})^{2}}{|x|^{2}}dx \geq 0.
\end{align*}
It is obvious that if $ \lambda \leq0$ then $ 0 \leq -\int_{\Omega}|\nabla v^{-}|^{2}dx$ and we deduce that $ \int_{\Omega} |\nabla v^{-}|^{2}dx = 0 $. 

If $ 0 \leq \lambda < \frac{N-2}{2}$ applying Hardy inequality we have 
$$ 0\leq -\int_{\Omega}|\nabla v^{-}|^{2}dx + \frac{\lambda}{2}(N-2)\int_{\Omega}\frac{(v^{-})^{2}}{|x|^{2}}dx \leq -\left[\frac{-2\lambda+N-2 }{N-2}\right]\int_{\Omega}|\nabla v^{-}|^{2}dx.$$
Since $\lambda < \frac{N-2}{2} $ then $ \frac{-2\lambda+N-2 }{N-2} >0$ which means that $ \int_{\Omega} |\nabla v^{-}|^{2}dx = 0 $. 

Therefore  we have $ v^{-}=0 $ and $ v \geq 0$ in $ \Omega$. 

\end{proof}

\begin{proof}[Proof of Theorem \eqref{Teorema 1}]
We split the proof into six steps.

\textbf{Step 1} We show that $$ \|v\|_{H^{1}_{0}(\Omega)} \leq C \|f\|_{L^{2}(\Omega)}.$$
Indeed, we take as test function $ \phi = v \in H^{1}_{0}(\Omega)$ in \eqref{20}.

The integration by parts leads to  
$$ \int_{\Omega}|\nabla v|^{2}dx - \lambda \frac{N-2}{2} \int_{\Omega} \frac{ v^{2}}{\|x\|^{2}}  dx = \int_{\Omega}fv dx .$$

Applying Hardy inequality, Cauchy-Schwartz inequality and Poincar\'e inequality we obtain  for $ \lambda < \frac{N-2}{2}$ a constant $ C(\lambda,N)>0$  such that 
\begin{align*}
    \|\nabla v\|^{2}_{L^{2}({\Omega})} &\leq \|f\|_{L^{2}(\Omega)}\|v\|_{L^{2}(\Omega)} \\
    & \leq C(\lambda,N) \|f\|_{L^{2}(\Omega)}\|\nabla v\|_{L^{2}(\Omega)}.
\end{align*}
Therefore we obtain the estimate $ \|v\|_{H^{1}_{0}(\Omega)} \leq C(\lambda,N) \|f\|_{L^{2}(\Omega)}$.

\textbf{Step 2} We want to construct an approximation Dirichlet problem for our case. We take a cut-off function $ \theta \in C^{\infty}([0,\infty))$, where $ 0 \leq \theta\leq 1$ such that 
\begin{align}\label{eq2a}
 \theta(x):=  \left\{\begin{array}{ll} 
0,  &		\textrm{ for} \quad  x< 1\\
	1, &  \textrm{ for }  x\geq 2,\\ 
	\end{array}\right.
\end{align}

We define $ \theta_{\varepsilon}(x)= \theta\left(\frac{|x|}{\varepsilon}\right)$ and we have that $ \theta_{\varepsilon}(x) = 0 $ in $ B_{\varepsilon}(0)$ with $ \theta_{\varepsilon}(x) = 1$ in $ \Omega \setminus B_{2\varepsilon}(0)$. It is clear that $ \theta_{\varepsilon} \rightarrow 1$ a.e and also in $ L^{2}(\Omega)$-norm when $ \varepsilon \rightarrow 0$.  

We set $ f_{\varepsilon}:= f\theta_{\varepsilon} \in L^{2}(\Omega)$ and we take the following Dirichlet problem 
\begin{align}\label{eq101}
  \left\{\begin{array}{ll} 
-\Delta v_{\varepsilon} + \lambda \frac{x\nabla v_{\varepsilon}}{|x|^{2}}= f_{\varepsilon}(x),  &		x\textrm{ in} \quad \Omega\\
v_{\varepsilon}(x)=	0, & x \textrm{ on }  \partial \Omega.\\ 
	\end{array}\right.
\end{align}

By Lax-Milgram theorem we know that \eqref{eq100} admits a unique solution $ v_{\varepsilon} \in H^{1}_{0}(\Omega)$ for $ \lambda < \frac{N-2}{2}$ and we have 
\begin{align}\label{eq102b}
  \left\{\begin{array}{ll} 
\int_{\Omega}\left[\nabla v_{\varepsilon}\cdot\nabla \phi + \lambda \frac{x\nabla v_{\varepsilon}}{|x|^{2}}\phi\right] dx = \int_{\Omega}f_{\varepsilon}\phi dx,  \quad\forall \phi \in H^{1}_{0}(\Omega) \\
v_{\varepsilon} \in H_{0}^{1}(\Omega).	\\ 
	\end{array}\right.
\end{align}

Since $ f_{\varepsilon} =0 $ in $ B_{\varepsilon}(0)$ we have that $$ \int_{B_{\varepsilon}(0)}\left[\nabla v_{\varepsilon}\cdot\nabla \phi + \lambda \frac{x\nabla v_{\varepsilon}}{|x|^{2}}\phi\right] dx =0 , $$
for any $ \phi \in H^{1}_{0}(B_{\varepsilon}(0))$. 

We restrict our attention to the region $ B_{\varepsilon}(0)\setminus \{0\}$ and we consider $ \phi \in C^{\infty}_{c}((B_{\varepsilon}(0)\setminus\{0\}))$. In this case the potential $ \frac{x}{|x|^{2}}$ is bounded in any compact included in $ B_{\varepsilon}(0) \setminus \{0\}$. 

In fact, by standard elliptic regularity we have that $ v_{\varepsilon} \in C^{\infty}(B_{\varepsilon}(0)\setminus\{0\})$. Indeed, we could consider $0< \varepsilon_{0}< \varepsilon$. 

Thus we get 
\begin{align}\label{eq101b}
  \left\{\begin{array}{ll} 
\int_{B_{\varepsilon}(0)}\left[\nabla v_{\varepsilon}\cdot\nabla \phi + \lambda \frac{x\nabla v_{\varepsilon}}{|x|^{2}}\phi\right] dx =0 ,  \quad\forall \phi \in H^{1}_{0}(B_{\varepsilon}(0)\setminus B_{\varepsilon_{0}}(0))\\
v_{\varepsilon} \in H^{1}(B_{\varepsilon}(0)\setminus B_{\varepsilon_{0}}(0))	\\ 
	\end{array}\right.
\end{align}

Since $ \frac{1}{|x|} \in C^{\infty}(B_{\varepsilon}(0)\setminus B_{\varepsilon_{0}}(0))$ we can apply \cite[Ch. 6, Th. 3]{Charro} and we obtain that $ v_{\varepsilon} \in C^{\infty}(B_{\varepsilon}(0)\setminus B_{\varepsilon_{0}}(0))$, but $ \varepsilon_{0}$ was chosen arbitrary and therefore we can conclude that $ v_{\varepsilon} \in C^{\infty}(B(0,\varepsilon)\setminus\{0\})$. 

Taking  $ \phi \in C^{\infty}_{c}(B_{\varepsilon}(0)\setminus{0})$ and integrating by parts \eqref{eq101} we obtain that $$ \int_{B_{\varepsilon}(0)\setminus\{0\}}\left(-\Delta v_{\varepsilon} + \lambda \frac{x\nabla v_{\varepsilon}}{|x|^{2}}\right)\phi(x)dx =0.$$
Since $ -\Delta v_{\varepsilon} + \lambda \frac{x\nabla v_{\varepsilon}}{|x|^{2}} \in L^{1}_{loc}(B_{\varepsilon}(0)\setminus{0})$ we can use the fundamental lemma of the calculus of variations and hence we get that $$-\Delta v_{\varepsilon} + \lambda \frac{x\nabla v_{\varepsilon}}{|x|^{2}} =0,  $$ where $ v_{\varepsilon} \in C^{\infty}(B_{\varepsilon}(0)\setminus \{0\})$. 

\textbf{Step 3} We prove that $ v_{\varepsilon} \in L^{\infty}(B_{\delta}(0)),$ where $ \delta < \varepsilon$. 

By \textbf{step 2} we know that $ v_{\varepsilon} \in C^{\infty}(B_{\varepsilon}(0)\setminus\{0\})$ and we have that $ v_{\varepsilon}$ is bounded on $ \partial B(0,\delta)$ for any $ \delta < \varepsilon$. 

We consider  the following Dirichlet problem
\begin{align}
    \left\{\begin{array}{cc} 
-\Delta (\|v_{\varepsilon}\|_{L^{\infty}(\partial B_{\delta}(0))} - v_{\varepsilon}) +\lambda\frac{x\cdot \nabla (\|v_{\varepsilon}\|_{L^{\infty}(\partial B_{\delta}(0))} - v_{\varepsilon})}{|x|^{2}} = 0,  &		\textrm{ in} \quad  B_{\delta}(0) \\
	\|v_{\varepsilon}\|_{L^{\infty}(\partial B_{\delta}(0))} - v_{\varepsilon}\ge0, &  \textrm{ on }  \partial B_{\delta}(0),\\ 
	\end{array}\right.
\end{align}
where $\|v_{\varepsilon}\|_{L^{\infty}(\partial B_{\delta}(0))} - v_{\varepsilon} \in H^{1}(B_{\delta}(0))\cap C^{\infty}(B_{\delta}(0) \setminus \{0\})$ and $ \lambda <\frac{N-2}{2}$. 

Applying \ref{ Teorema 2}  we get that $ \|v_{\varepsilon}\|_{L^{\infty}(\partial B_\delta(0))} - v_{\varepsilon} \geq 0$ and we deduce that $$ \|v_{\varepsilon}\|_{L^{\infty}(B_{\delta}(0))} \leq \|v_{\varepsilon}\|_{L^{\infty}(\partial B_{\delta}(0))}.$$ Hence we have that $ v_{\varepsilon} \in L^{\infty}(B_{\delta}(0))$. 

\textbf{Step 4} We show that $ v_{\varepsilon} \in H^{2}(\Omega)$. 

 We multiply by $ |x|^{-\lambda}$ the equation $ -\Delta v_{\varepsilon}+ \lambda \frac{x\nabla v_{\varepsilon}}{|x|^{2}}=0$ on $ B_{\varepsilon}(0) \setminus \{0\}$ and we get that 
$$ - \Delta v_{\varepsilon} |x|^{-\lambda} + \lambda x\cdot \nabla v _{\varepsilon}|x|^{-\lambda -2} = 0, $$
which leads at the following divergence form 
$$ - |x|^{\lambda}\mathrm{div}(|x|^{-\lambda}\nabla v_{\varepsilon}) = 0 $$
Since $ v_{\varepsilon} \in C^{\infty}(B_{\varepsilon}(0)\setminus\{0\})$ we avoid the singularity and we have 
\begin{align}\label{eq104}
    \mathrm{div}(|x|^{-\lambda}\nabla v_{\varepsilon}) = 0.
\end{align}
We choose  the sequence of functions 
$ (\psi_{n})_{n \in \mathbb{N}}\subset C_{c}^{\infty}(\mathbb{R}^{N})$ such that

\begin{align}
 \psi_{n}(x)=   \left\{\begin{array}{cc} 
0,  &		\textrm{ for}\quad |x| \leq \frac{1}{n} \ \\
	1, &  \textrm{ for }\quad  \frac{2}{n}\leq|x| \leq \frac{\delta}{4}\\   0    &  \textrm{ for }\quad  |x| \geq \frac{\delta}{2}\\ 
	\end{array}\right.
\end{align}

We multiply the equation \eqref{eq104} by the function $ \psi_{n}|x|^{\lambda -2}v_{\varepsilon} \in C_{c}^{\infty}(B_{\delta}(0)\setminus \{0\})$. Integrating by parts we get that 
$$ 0 = \int_{B_{\delta}(0)}\mathrm{div}(|x|^{-\lambda}\nabla v_{\varepsilon})\psi_{n}|x|^{\lambda -2}v_{\varepsilon}dx= -\int_{B_{\delta}(0)}|x|^{-\lambda}\nabla v_{\varepsilon}\mathrm{div}(\psi_{n}|x|^{\lambda-2}v_{\varepsilon})dx$$
A direct computation yields 
$$ 0 = - \int_{B_{\delta}(0)}\frac{|\nabla v_{\varepsilon}|^{2}}{|x|^{2}}\psi_{n}(x)dx - (\lambda -2)\int_{B_{\delta}(0)}v_{\varepsilon}\nabla v_{\varepsilon}\psi_{n}\frac{x}{|x|^{4}}dx -\int_{B_{\delta}(0)}|x|^{-2}\nabla \psi_{n}v_{\varepsilon}\nabla v_{\varepsilon}dx $$
which is equivalent with 
\begin{align}\label{127}
    \int_{B_{\delta}(0)}\frac{|\nabla v_{\varepsilon}|^{2}}{|x|^{2}}\psi_{n}(x)dx = - (\lambda -2)\int_{B_{\delta}(0)}v_{\varepsilon}\nabla v_{\varepsilon}\psi_{n}\frac{x}{|x|^{4}}dx -\int_{B_{\delta}(0)}|x|^{-2}\nabla \psi_{n}v_{\varepsilon}\nabla v_{\varepsilon}dx
\end{align}
We take each term and using the definition of function $ \psi_{n}$ we get 
\begin{align*}
  \int_{B_{\delta}(0)}v_{\varepsilon}\nabla v_{\varepsilon}\psi_{n}\frac{x}{|x|^{4}}dx &= \int_{B_{\frac{2}{n}}(0)\setminus B_{\frac{1}{n}}(0)} v_{\varepsilon}\nabla v_{\varepsilon}\psi_{n}\frac{x}{|x|^{4}}dx+\int_{B_{\frac{\gamma}{4}}(0)\setminus B_{\frac{2}{n}}(0)} v_{\varepsilon}\nabla v_{\varepsilon}\frac{x}{|x|^{4}}dx \\
  &+ \int_{B_{\frac{\gamma}{4}}(0)\setminus B_{\frac{\gamma}{2}}(0)} v_{\varepsilon}\nabla v_{\varepsilon}\psi_{n}\frac{x}{|x|^{4}}dx
\end{align*}
We analyze each term and we have 
\begin{align*}
    \int_{B_{\frac{2}{n}}(0)\setminus B_{\frac{1}{n}}(0)} v_{\varepsilon}\nabla v_{\varepsilon}\psi_{n}\frac{x}{|x|^{4}}dx &=\frac{1}{2}\int_{\partial B_{\frac{2}{n}}(0)}\frac{v_{\varepsilon}^{2}\psi_{n}}{|x|^{3}}dx + \frac{1}{2}\int_{\partial B_{\frac{1}{n}}(0)}\frac{v_{\varepsilon}^{2}\psi_{n}}{|x|^{3}}dx \\
    &- \frac{1}{2}\int_{B_{\frac{2}{n}}(0)\setminus B_{\frac{1}{n}}(0)} v_{\varepsilon}^{2}\nabla \left(\psi_{n}\frac{x}{|x|^{4}}\right)dx 
\end{align*}
Keeping into account that $ v_{\varepsilon} \in L^{\infty}(B_{\delta}(0)) $ and $ 0 \leq  \psi_{n}\leq 1$ we have that 
\begin{align}\label{eq105}
 \left| \int_{\partial B_{\frac{2}{n}}(0)}\frac{v_{\varepsilon}^{2}\psi_{n}}{|x|^{3}}dx\right| \leq \|v_{\varepsilon}\|^{2}_{L^{\infty}(\partial B_{\frac{2}{n}}(0))}\left(\frac{2}{n}\right)^{N-4}  
\end{align} 
Similarly, we obtain that 
\begin{align}\label{eq106}
    \left|\int_{\partial B_{\frac{1}{n}}(0)}\frac{v_{\varepsilon}^{2}\psi_{n}}{|x|^{3}}dx\right| \leq \|v_{\varepsilon}\|^{2}_{L^{\infty}(\partial B_{\frac{1}{n}}(0))}\left(\frac{1}{n}\right)^{N-4}
\end{align}
We observe that 
\begin{align*}
    \int_{B_{\frac{2}{n}}(0)\setminus B_{\frac{1}{n}}(0)} v_{\varepsilon}^{2}\nabla \left(\psi_{n}\frac{x}{|x|^{4}}\right)dx = \int_{B_{\frac{2}{n}}(0)\setminus B_{\frac{1}{n}}(0)}v_{\varepsilon}^{2}\frac{x\nabla \psi_{n}}{|x|^{4}}dx+ (N-2) \int_{B_{\frac{2}{n}}(0)\setminus B_{\frac{1}{n}}(0)}\frac{v_{\varepsilon}^{2}\psi_{n}}{|x|^{4}}dx
\end{align*}
We compute and for $ N \geq 4$ we get that
\begin{align}\label{eq107}
  \left|\int_{B_{\frac{2}{n}}(0)\setminus B_{\frac{1}{n}}(0)}v_{\varepsilon}^{2}\frac{x\nabla \psi_{n}}{|x|^{4}}dx\right|\leq n\|v_{\varepsilon}\|^{2}_{L^{\infty}(B_{\delta}(0))}\int_{B_{\frac{2}{n}}(0)\setminus B_{\frac{1}{n}}(0)}\frac{1}{|x|^{3}}dx \leq  \|v_{\varepsilon}\|^{2}_{L^{\infty}(B_{\delta}(0))}\left(\frac{2^{N-3}-1}{n^{N-3}}\right)
\end{align}

Also for $ N \geq 5$ we have 
\begin{align}\label{eq108}
   \left| \int_{B_{\frac{2}{n}}(0)\setminus B_{\frac{1}{n}}(0)}v_{\varepsilon}^{2}\frac{ \psi_{n}}{|x|^{4}}dx\right| \leq \|v_{\varepsilon}\|^{2}_{L^{\infty}(B_{\delta}(0))}\left(\frac{2^{N-4}-1}{n^{N-4}}\right)
\end{align}

Inequalities \eqref{eq105},\eqref{eq106},\eqref{eq107} and \eqref{eq108}
 lead to 
\begin{align*}
   \left|\int_{B_{\frac{2}{n}}(0)\setminus B_{\frac{1}{n}}(0)}v_{\varepsilon}\nabla v_{\varepsilon}\psi_{n}\frac{x}{|x|^{4}}dx\right| \leq C_{1}(\varepsilon,N)\left(\frac{1}{n^{N-4}}+\frac{1}{n^{N-3}}\right), 
\end{align*}
where $ C_{1}(\varepsilon,N)>0$ and $ N \geq 5$. 

In the same way we can control the quantities $$\left|\int_{B_{\frac{\delta}{4}}(0)\setminus B_{\frac{2}{n}}(0)} v_{\varepsilon}\nabla v_{\varepsilon}\frac{x}{|x|^{4}}dx\right| \quad \textrm{and}\quad\left|\int_{B_{\frac{\delta}{2}}(0)\setminus B_{\frac{\delta}{4}}(0)} v_{\varepsilon}\nabla v_{\varepsilon}\psi_{n}\frac{x}{|x|^{4}}dx\right|.$$ 
Putting all together we deduce that for $ N \geq 5$ there exists $ \tilde C_{1}(\varepsilon,N)>0$ such that 
$$ \left|\int_{B_{\delta}(0)}v_{\varepsilon}\nabla v_{\varepsilon}\psi_{n}\frac{x}{|x|^{4}}dx\right| \leq\tilde C_{1}(\varepsilon,N) \left(\frac{1}{n^{N-4}} +\frac{1}{n^{N-3}}+ \delta^{N-4}+ \delta^{N-3}\right) $$

On the other hand, we estimate the last term 
$$ \int_{B_{\delta}(0)}|x|^{-2}\nabla \psi_{n}v_{\varepsilon}\nabla v_{\varepsilon}dx= \int_{B_{\frac{2}{n}}(0)\setminus B_{\frac{1}{n}}(0)} |x|^{-2}\nabla \psi_{n}v_{\varepsilon}\nabla v_{\varepsilon}dx + \int_{B_{\frac{\delta}{2}}(0)\setminus B_{\frac{\delta}{4}}(0)} |x|^{-2}\nabla \psi_{n}v_{\varepsilon}\nabla v_{\varepsilon}dx  $$

Using the integrating by parts we get that 
\begin{align*}
   \int_{B_{\frac{2}{n}}(0)\setminus B_{\frac{1}{n}}(0)} |x|^{-2}\nabla \psi_{n}v_{\varepsilon}\nabla v_{\varepsilon}dx
   & =-\frac{1}{2}\int_{B_{\frac{2}{n}}(0)\setminus B_{\frac{1}{n}}(0)} v_{\varepsilon}^{2}\mathrm{div}(|x|^{-2}\nabla \psi_{n})dx +\frac{1}{2}\int_{\partial B_{\frac{2}{n}}(0)} v_{\varepsilon}^{2} |x|^{-2}\nabla \psi_{n}dx \\
   &+\frac{1}{2} \int_{\partial B_{\frac{1}{n}}(0)} v_{\varepsilon}^{2} |x|^{-2}\nabla \psi_{n}dx   
\end{align*}
We estimate each integral and we have 
\begin{align*}
  \left | \int_{\partial B_{\frac{2}{n}}(0)} v_{\varepsilon}^{2} |x|^{-2}\nabla \psi_{n}dx\right| \leq \omega_{N}2^{N-3}\frac{\|v_{\varepsilon}\|^{2}_{L^{\infty}(B_{\delta}(0))}}{n^{N-4}}
\end{align*}
Also we get 
\begin{align*}
    \left|\int_{\partial B_{\frac{1}{n}}(0)} v_{\varepsilon}^{2} |x|^{-2}\nabla \psi_{n}dx\right| \leq \omega_{N}\frac{\|v_{\varepsilon}\|^{2}_{L^{\infty}(B_{\delta}(0))}}{n^{N-4}}
\end{align*}
By computation we reach that for $ N \geq 5$
\begin{align*}
    \left|\int_{B_{\frac{2}{n}}(0)\setminus B_{\frac{1}{n}}(0)} v_{\varepsilon}^{2}\mathrm{div}(|x|^{-2}\nabla \psi_{n})dx\right|& \leq \int_{B_{\frac{2}{n}}(0)\setminus B_{\frac{1}{n}}(0)}|v_{\varepsilon}^{2}|x|^{-2}\Delta \psi_{n}|dx + \int_{B_{\frac{2}{n}}(0)\setminus B_{\frac{1}{n}}(0)} |v_{\varepsilon}^{2}\nabla (|x|^{-2})\nabla \psi_{n}(x)|dx \\
    &\leq \|v_{\varepsilon}\|_{L^{\infty}(B_{\delta}(0))}^{2}n^{2}\int_{B_{\frac{2}{n}}(0)\setminus B_{\frac{1}{n}}(0)}|x|^{-2}dx + \\
    &+2\|v_{\varepsilon}\|_{L^{\infty}(B_{\delta}(0))}^{2}n\int_{B_{\frac{2}{n}}(0)\setminus B_{\frac{1}{n}}(0)}|x|^{-3}dx \\
    & \leq \|v_{\varepsilon}\|_{L^{\infty}(B_{\delta}(0))}^{2} \left(\frac{2^{N-2}-1}{n^{N-4}} + \frac{2^{N-2}-2}{n^{N-4}}\right)
\end{align*}
Collecting all estimates we obtain that 
$$\left|\int_{B_{\frac{2}{n}}(0)\setminus B_{\frac{1}{n}}(0)} |x|^{-2}\nabla \psi_{n}v_{\varepsilon}\nabla v_{\varepsilon}dx\right| \leq \frac{C_{2}(\varepsilon,N)}{n^{N-4}},$$
where $ C_{2}(\varepsilon,N)>0$. 

Making the same argument we have 
$$\left|\int_{B_{\frac{\delta}{2}}(0)\setminus B_{\frac{\delta}{4}}(0)} |x|^{-2}\nabla \psi_{n}v_{\varepsilon}\nabla v_{\varepsilon}dx\right| \leq C_{2}(\varepsilon,N)\delta^{N-4}.$$

Coupling all inequalities we get that 
$$\left |\int_{B_{\delta}(0)}\frac{|\nabla v_{\varepsilon}|^{2}}{|x|^{2}}\psi_{n}(x)dx\right| \leq C(\varepsilon,N,\lambda)\left(\frac{1}{n^{N-4}}+ \frac{1}{n^{N-3}}+ \delta^{N-4}+ \delta ^{N-3}\right).$$
Applying the Fatou lemma and keeping into account that $ N \geq 5$ we reach that 
\begin{align}\label{eq116}
    \int_{B_{\delta}(0)}\frac{|\nabla v_{\varepsilon}|^{2}}{|x|^{2}}dx \leq \lim_{n\rightarrow \infty} \inf \int_{B(0,\delta)} \frac{|\nabla v_{\varepsilon}|^{2}}{|x|^{2}}\psi_{n}dx \leq C(\varepsilon,N,\lambda)(\delta ^{N-4}+\delta ^{N-3})< \infty.
\end{align}
We know that 
\begin{align}
  \left\{\begin{array}{ll} 
-\Delta v_{\varepsilon}=  f_{\varepsilon}-\lambda \frac{x\nabla v_{\varepsilon}}{|x|^{2}}  &		\textrm{ in} \quad   B_{\delta}(0)\\
v_{\varepsilon} \in H^{1}(B_{\delta}(0)).
 	\end{array}\right.
\end{align}
By virtute of \eqref{eq116} and $ f_{\varepsilon} \in L^{2}(B_{\delta}(0)) $ we deduce that $ f_{\varepsilon} - \lambda\frac{x\nabla v_{\varepsilon}}{|x|^{2}} \in L^{2}(B_{\delta}(0))$ and by standard elliptic regularity given by Calderon Zygmund we have $ v_{\varepsilon} \in H^{2}(B_{\delta}(0))$ for any $ \varepsilon >0$.  

It is well-known that $ v_{\varepsilon} \in H^{2}(\Omega \setminus B_{\delta}(0))$ and therefore it is clear that $ v_{\varepsilon} \in H^{2}(\Omega)$.

\textbf{Step 5} We prove that $ v_{\varepsilon} \rightarrow v$ in $ H_{0}^{1}(\Omega)$

Indeed, we have by linearity of the problem that
$$ \int_{\Omega} \nabla (v_{\varepsilon}-v)\nabla\phi + \lambda \frac{x}{|x|^{2}}\nabla(v_{\varepsilon}-v)\phi dx= \int_{\Omega} (f_{\varepsilon}-f)\phi dx,$$
for any $ \phi \in H^{1}_{0}(\Omega)$.

Taking $ \phi = v_{\varepsilon}-v \in H^{1}_{0}(\Omega)$ and using integration by parts, Hardy inequality and Cauchy-Schwartz inequality there exists $ C(\lambda, N)$ such that 
$$\|\nabla(v_{\varepsilon}-v)\|_{L^{2}(\Omega)}^{2}dx \leq C(\lambda,N)\|f_{\varepsilon}-f\|_{L^{2}(\Omega)}\|v_{\varepsilon}-v\|_{L^{2}(\Omega)}$$

Since $ \Omega$ is bounded,  Poincar\'e inequality yields 
$$ \|(v_{\varepsilon}-v)\|_{H^{1}_{0}(\Omega)}\leq C(\lambda,N)\|f_{\varepsilon}-f\|_{L^{2}(\Omega)}.$$
This means that $$ \partial_{xi}v_{\varepsilon} \rightarrow \partial_{x_{i}}v$$ in $ L^{2}(\Omega)$ when $ \varepsilon \rightarrow 0$ because $ f_{\varepsilon} \rightarrow f$.

In fact we know that $ \partial_{x_{i}}v_{\varepsilon} \rightarrow \partial_{x_{i}}v$ a.e in $ \Omega$ up to a subsequence. 

\textbf{Step 6} We show that $v \in H^{2}(\Omega)$. 

By density we can extend the inequality \eqref{HR_lambda} to $ H_{0}^{2}(\Omega)$ when $ N \geq 5$. 
In fact we can't apply directly this result for $ v_{\varepsilon}$ since it does not belong in $ H_{0}^{2}(\Omega)$. 

We will use a cut-off argument. We take a function $ \eta \in C^{\infty}_{c}(\Omega)$, where $ 0 \leq \eta\leq 1$ such that 
\begin{align}\label{eq2b}
 \eta(x):=  \left\{\begin{array}{ll} 
1,  &		\textrm{ for} \quad  |x|< \frac{1}{2}\\
	0, &  \textrm{ for }  |x|\geq 1,\\ 
	\end{array}\right.
\end{align} 

We evaluate 
\begin{align*}
\int_{\Omega}\frac{|x\nabla v_{\varepsilon}|^{2}}{|x|^{4}}dx \leq \int_{\Omega} \frac{|x \nabla (v_{\varepsilon}\eta)|^{2}}{|x|^{4}}dx + \int_{\Omega \setminus B_{\frac{1}{2}}(0)}\frac{|x \nabla(v_{\varepsilon}(1-\eta))|^{2}}{|x|^{4}}dx  
\end{align*}

On the other hand, $ v_{\varepsilon} \eta \in H^{2}_{0}(\Omega)$ and we know that for $\lambda \in \rr\setminus\left\{\lambda_n:=\frac{(2n+N-2)^2-4} {2(N-4)} \  | n \in \nn^\star \right\} \setminus \{\frac{N}{2}\}$ we can apply \eqref{HR_lambda}  which holds for every $ \lambda < \frac{N-2}{2} $. 

Hardy-Rellich inequality \eqref{HR_ineq1}  gives us that 
\begin{align*}
    \int_{\Omega}\frac{|x\nabla(v_{\varepsilon}\eta)|^{2}}{|x|^{4}}dx \leq C(\lambda, N) \int_{\Omega}\left|-\Delta (v_{\varepsilon}\eta) + \lambda \frac{x\nabla (v_{\varepsilon}\eta)}{|x|^{2}}\right|^{2}dx.
\end{align*}
Putting $ v_{\varepsilon}\eta = v_{\varepsilon} + v_{\varepsilon}(\eta -1) $  there exists $ C>0$ such that 
\begin{align}\label{eq133}
    C \int_{\Omega}\left|-\Delta (v_{\varepsilon}\eta) + \lambda \frac{x\nabla (v_{\varepsilon}\eta)}{|x|^{2}}\right|^{2}dx &\leq \int_{\Omega}\left|- \Delta v_{\varepsilon} + \lambda \frac{x\nabla v_{\varepsilon}}{|x|^{2}}\right|^{2}  dx +\\ \nonumber 
     & + \int_{\Omega \setminus B_{\frac{1}{2}(0)}}\left|- \Delta [v_{\varepsilon}(\eta-1)] + \lambda \frac{x\nabla (v_{\varepsilon}(\eta -1))}{|x|^{2}} \right|^{2}dx \\ \nonumber  
     & \leq \int_{\Omega}\left|- \Delta v_{\varepsilon} + \lambda \frac{x\nabla v_{\varepsilon}}{|x|^{2}}\right|^{2}  dx + \int_{\Omega \setminus B_{\frac{1}{2}(0)} }\left|- \Delta v_{\varepsilon} + \lambda \frac{x\nabla v_{\varepsilon}}{|x|^{2}}\right|^{2}(\eta -1)^{2}dx \\ \nonumber 
     &+\int_{\Omega  \setminus B_{\frac{1}{2}}(0)}\left|-2\nabla v_{\varepsilon}\nabla \eta - \Delta \eta v_{\varepsilon}  + \lambda \frac{x\nabla{\eta}}{|x|^{2}}v_{\varepsilon}\right|^{2}dx
\end{align}

Using Hardy inequality and Cauchy-Schwartz inequality we can estimate the last term of the previous inequality and we get that 

\begin{align*}
    \int_{\Omega \setminus B_{\frac{1}{2}}(0)} \left|-2\nabla v_{\varepsilon}\nabla \eta - \Delta \eta v_{\varepsilon}  + \lambda \frac{x\nabla{\eta}}{|x|^{2}}v_{\varepsilon}\right|^{2}dx &\leq 4\|\nabla v_{\varepsilon}\|_{L^{2}(\Omega \setminus B_{\frac{1}{2}(0)})}\|\nabla \eta\|_{L^{\infty}(\Omega)}+ \|\Delta \eta\|_{L^{\infty}(\Omega)}\|v_{\varepsilon}\|_{L^{2}(\Omega \setminus B_\frac{1}{2})}+ \\
    &+\lambda^{2}\left\|\frac{v_{\varepsilon}}{|x|}\right\|_{L^{2}(\Omega \setminus B_{\frac{1}{2}})}\|\nabla \eta\|_{L^{\infty}(\Omega)}
\end{align*}  
As in \textbf{step 1} we argue that $ \|v_{\varepsilon}\|_{L^{2}(\Omega \setminus B_{\frac{1}{2}}(0))} \leq \|f\|_{L^{2}(\Omega)}$, $ \|\nabla v_{\varepsilon}\|_{L^{2}(\Omega \setminus B_{\frac{1}{2}(0)})} \leq \|f\|_{L^{2}(\Omega )}$ and $ \frac{1}{|x|^{2}} \leq 4$ for $ x \in \Omega \setminus B_{\frac{1}{2}}$ we have that  
\begin{align*}
    \int_{\Omega \setminus B_{\frac{1}{2}}(0)} \left|-2\nabla v_{\varepsilon}\nabla \eta - \Delta \eta v_{\varepsilon}  + \lambda \frac{x\nabla{\eta}}{|x|^{2}}v_{\varepsilon}\right|^{2}dx &\leq 4\|f\|_{L^{2}(\Omega )}\|\nabla \eta\|_{L^{\infty}(\Omega)}+ \|\Delta \eta\|_{L^{\infty}(\Omega)}\|f\|_{L^{2}(\Omega )}+ \\
    &+4\lambda^{2}\|f \|_{L^{2}(\Omega)}\|\nabla \eta\|_{L^{\infty}(\Omega)}.
\end{align*} 

Since $ - \Delta v_{\varepsilon} + \lambda \frac{x\nabla v_{\varepsilon}}{|x|^{2}} = f_{\varepsilon} $ the inequality  \eqref{eq133} becomes 
$$\int_{\Omega}\left|-\Delta (v_{\varepsilon}\eta) + \lambda \frac{x\nabla (v_{\varepsilon}\eta)}{|x|^{2}}\right|^{2}dx \leq C(\lambda^{2}, |\nabla  \eta|_{L^{\infty}(\Omega)}, |\Delta \eta|_{L^{\infty}(\Omega)},N)\|f\|_{L^{2}(\Omega)},$$ where $C(\lambda^{2}, |\nabla  \eta|, |\Delta \eta|, N)>0$. 

By a direct computation we have that 
\begin{align*}
    \int_{\Omega \setminus B_{\frac{1}{2}}(0)}\frac{|x \nabla (v_{\varepsilon}(1-\eta))|^{2}}{|x|^{4}}dx&\leq \|(1-\eta)^{2}\|_{L^{\infty}(\Omega)}\int_{\Omega \setminus B_{\frac{1}{2}(0)}}\frac{|\nabla v_{\varepsilon}|^{2}}{|x|^{2}}dx + \||\nabla \eta|^{2}\|_{L^{\infty}(\Omega)}\int_{\Omega \setminus B_{\frac{1}{2}(0)}}\frac{v_{\varepsilon}^{2}}{|x|^{2}}dx \\
    & \leq (4\|(1-\eta)^{2}\|_{L^{\infty}(\Omega)}+ 4\||\nabla \eta|^{2}\|_{L^{\infty}(\Omega)})\|f\|_{L^{2}(\Omega)}
\end{align*}

Putting all together we can estimate 
\begin{align*}
    \int_{\Omega}\frac{|x\nabla v_{\varepsilon}|^{2}}{|x|^{4}}dx &\leq \int_{\Omega}\frac{|x\nabla (v_{\varepsilon}\eta)|^{2}}{|x|^{4}}dx + \int_{\Omega \setminus B_{\frac{1}{2}(0)}}\frac{|x\nabla(v_{\varepsilon}(1-\eta))|^{2}}{|x|^{4}}dx \nonumber \\
    & \leq \tilde{C}(\lambda, N) \int_{\Omega}|f|^{2}dx,
\end{align*}
where $ \tilde{C}(\lambda, N)>0 $. 

We proved that the integral $ \int_{\Omega}\frac{|x\nabla v_{\varepsilon}|^{2}}{|x|^{4}}dx $ is uniformly bounded in $ \varepsilon$. 

By \textbf{step 5} Fatou lemma implies that 
$$ \int_{\Omega} \frac{|x\nabla v|^{2}}{|x|^{4}}dx \leq \lim_{\varepsilon \rightarrow 0} \inf \int_{\Omega} \frac{|x\nabla v_{\varepsilon}|^{2}}{|x|^{4}}  \leq \tilde{C}(\lambda,N) \int_{\Omega}|f|^{2}dx < \infty$$ 
This is sufficient to deduce that $ v \in H^{2}(B_{\delta}(0))$ and therefore the conclusion follows. 
\end{proof}

We formulate a version of maximum principle for Hardy-Leray operator.  
\begin{prop}\label{ Teorema 7}
    Let $ \Omega$ denote a smoothly bounded domain of $ \mathbb{R}^{N}$, $ N \ge 3$. Assume that $ v \in H^{1}( \Omega) $  satisfies, in the weak sense,
\begin{align}\label{eq120}
    \left\{\begin{array}{cc} 
-\Delta v +\lambda\frac{ v}{|x|^{2}} = 0,  &		\textrm{ in} \quad  \Omega \\
	v\ge0, &  \textrm{ on }  \partial \Omega,\\ 
	\end{array}\right.
\end{align}
where $ \lambda >  -\frac{(N-2)^{2}}{4}$. Then $ v \ge 0$ in $ \Omega$. 
\end{prop}
\begin{proof}[Proof of Proposition \eqref{ Teorema 7}]

Let $ v$ denote a solution to \eqref{eq120}. We know that 
\begin{align}\label{eq118}
    \int_{\Omega} \left[\nabla v \cdot \nabla \phi + \lambda \frac{ v}{|x|^{2}}\phi\right]dx = 0,
\end{align}
for any $ \phi \in H_{0}^{1}(\Omega)$.

We set 
\begin{align}\label{eq119}
 v^{-}(x)= \left\{\begin{array}{ll} 
-v,  &		v\leq 0\\
	0, &  v\geq 0.\\ 
	\end{array}\right.
\end{align}
It is known that $ v^{-} \in H^{1}(\Omega)$ when $ v \in H^{1}(\Omega)$. Since $ v \geq 0$ on $ \partial \Omega$ we have $ v^{-} = 0$ on $ \partial \Omega$. 
We can set $ \phi:=v^{-} \in H_{0}^{1}(\Omega)$. Keeping into account that $ v = v^{+}-v^{-}$ the equality \eqref{eq118} becomes 
$$ - \int_{\Omega}|\nabla v^{-}|^{2}dx - \lambda\int_{\Omega}\frac{(v^{-})^{2}}{|x|^{2}}dx =  0.$$
Hardy inequality implies that 
$$ 0= \int_{\Omega}|\nabla v^{-}|^{2}dx + \lambda  \int_{\Omega}\frac{(v^{-})^{2}}{|x|^{2}} \geq \left[\frac{(N-2)^{2}}{4} + \lambda\right]\int_{\Omega} \frac{(v^{-})^{2}}{|x|^{2}} dx $$
Since $\lambda >-\frac{(N-2)^{2}}{4} $ then  we get $ \frac{(v^{-})^{2}}{|x|^{2}}=0$ and we have $ v^{-}=0 $. This means that $ v \geq 0$ in $ \Omega$. 

\end{proof}

\begin{proof}[Proof of Theorem \eqref{Teorema 2}]
We divide the proof into five steps.

\textbf{Step 1} We show that $$ \|v\|_{H^{1}_{0}(\Omega)} \leq C \|f\|_{L^{2}(\Omega)}.$$
Indeed, we take as test function $ \phi = v \in H^{1}_{0}(\Omega)$ in \eqref{def2}.

Hence we get 
$$ \int_{\Omega}|\nabla v|^{2}dx + \lambda \int_{\Omega} \frac{ v^{2}}{\|x\|^{2}}  dx = \int_{\Omega}fv dx .$$

Applying Hardy inequality, Cauchy-Schwartz inequality and Poincar\'e inequality we obtain that there exists $ C(\lambda,N)>0$ such that  
\begin{align*}
    \|\nabla v\|^{2}_{L^{2}({\Omega})} &\leq \|f\|_{L^{2}(\Omega)}\|v\|_{L^{2}(\Omega)} \\
    & \leq C(\lambda,N) \|f\|_{L^{2}(\Omega)}\|\nabla v\|_{L^{2}(\Omega)}.
\end{align*}
Therefore we have that   $\|v\|_{H^{1}_{0}(\Omega)} \leq C(\lambda,N) \|f\|_{L^{2}(\Omega)}$.

\textbf{Step 2} We want to construct an approximation Dirichlet problem for our case. We take a cut-off function $ \theta \in C^{\infty}([0,\infty))$, where $ 0 \leq \theta\leq 1$ such that 
\begin{align}\label{eq2b2}
 \theta(x):=  \left\{\begin{array}{ll} 
0,  &		\textrm{ for} \quad  x< 1\\
	1, &  \textrm{ for }  x\geq 2,\\ 
	\end{array}\right.
\end{align}

We define $ \theta_{\varepsilon}(x)= \theta\left(\frac{|x|}{\varepsilon}\right)$.  
Hence we have that $ \theta_{\varepsilon}(x) = 0 $ in $ B_{\varepsilon}(0)$ and $ \theta_{\varepsilon}(x) = 1$ in $ \Omega \setminus B_{2\varepsilon}(0)$. It is well known that $ \theta_{\varepsilon} \rightarrow 1$ a.e and also in $ L^{2}(\Omega)$-norm when $ \varepsilon \rightarrow 0$. 

We set $ f_{\varepsilon}:= f\theta_{\varepsilon} \in L^{2}(\Omega)$ and we take the following Dirichlet problem 
\begin{align}\label{eq121}
  \left\{\begin{array}{ll} 
-\Delta v_{\varepsilon} + \lambda \frac{ v_{\varepsilon}}{|x|^{2}}= f_{\varepsilon},  &		\textrm{ in} \quad \Omega\\
v_{\varepsilon}=	0, &  \textrm{ on } \quad \partial \Omega,\\ 
	\end{array}\right.
\end{align}

By Lax-Milgram theorem we know that \eqref{eq121} admits a unique solution $ v_{\varepsilon} \in H^{1}_{0}(\Omega)$ for $ \lambda >- \frac{N(N-4)}{4}$ and we have 
\begin{align}\label{eq122}
  \left\{\begin{array}{ll} 
\int_{\Omega}\nabla v_{\varepsilon}\cdot\nabla \phi + \lambda \frac{ v_{\varepsilon}}{|x|^{2}}\phi dx = \int_{\Omega}f_{\varepsilon}\phi dx,  \quad\forall \phi \in H^{1}_{0}(\Omega) \\
v_{\varepsilon} \in H_{0}^{1}(\Omega)	\\ 
	\end{array}\right.
\end{align}

Since $ f_{\varepsilon} =0 $ in $ B_{\varepsilon}(0)$ we have that $$ \int_{B_{\varepsilon}(0)}\nabla v_{\varepsilon}\cdot\nabla \phi + \lambda \frac{ v_{\varepsilon}}{|x|^{2}}\phi dx =0 , $$
for any $ \phi \in H^{1}_{0}(B_{\varepsilon}(0))$. 

We restrict our attention to the region $ B_{\varepsilon}(0)\setminus \{0\}$ and we consider $ \phi \in C^{\infty}_{c}((B_{\varepsilon}(0)\setminus\{0\}))$. In this case the potential $ \frac{1}{|x|^{2}}$ is bounded in any compact included in $ B_{\varepsilon}(0) \setminus \{0\}$. 

In fact, by standard elliptic regularity we have that $ v_{\varepsilon} \in C^{\infty}(B_{\varepsilon}(0)\setminus\{0\})$. Indeed, we could consider $0< \varepsilon_{0}< \varepsilon$. 

Thus we get 
\begin{align}\label{190}
  \left\{\begin{array}{ll} 
\int_{B_{\varepsilon}(0)}\nabla v_{\varepsilon}\cdot\nabla \phi + \lambda \frac{ v_{\varepsilon}}{|x|^{2}}\phi dx =0 ,  \quad\forall \phi \in H^{1}_{0}(B_{\varepsilon}(0)\setminus B_{\varepsilon_{0}}(0)) \\
v_{\varepsilon} \in H^{1}(B_{\varepsilon}(0)\setminus B_{\varepsilon_{0}}(0))	\\ 
	\end{array}\right.
\end{align}

Since $ \frac{1}{|x|} \in C^{\infty}(B_{\varepsilon}(0)\setminus B_{\varepsilon_{0}}(0))$ we can apply \cite[Ch. 6, Th. 3]{Charro} and we obtain that $ v_{\varepsilon} \in C^{\infty}(B_{\varepsilon}(0)\setminus B_{\varepsilon_{0}}(0))$, but $ \varepsilon_{0}$ was chosen arbitrary and therefore we can conclude that $ v_{\varepsilon} \in C^{\infty}(B_{\varepsilon}(0)\setminus\{0\})$. 

Taking $ \phi \in C^{\infty}_{c}(B_{\varepsilon}(0)\setminus{0})$ and integrating by parts \eqref{190}  we obtain that $$ \int_{B(0,\varepsilon)\setminus\{0\}}\left(-\Delta v_{\varepsilon} + \lambda \frac{ v_{\varepsilon}}{|x|^{2}}\right)\phi(x)dx =0.$$
Since $ -\Delta v_{\varepsilon} + \lambda \frac{ v_{\varepsilon}}{|x|^{2}} \in L^{1}_{loc}(B_{\varepsilon}(0)\setminus{0})$ we can apply the fundamental lemma of the calculus of variations and hence we get that $$-\Delta v_{\varepsilon} + \lambda \frac{ v_{\varepsilon}}{|x|^{2}} =0,  $$ where $ v_{\varepsilon} \in C^{\infty}(B_{\varepsilon}(0)\setminus \{0\})$. 

\textbf{Step 3} We prove that $ v_{\varepsilon} \in H^{2}(B_{\delta}(0))$. 

By \textbf{step 2} we know that $ v_{\varepsilon} \in C^{\infty}(B_{\varepsilon}(0)\setminus\{0\})$ and we have that $ v_{\varepsilon}$ is bounded on $ \partial B(0,\delta)$ for any $ \delta < \varepsilon$.

By \eqref{ Teorema 8} we know that $ u_{\lambda}(x)= |x|^{-\frac{N-2}{2}+ \sqrt{ \lambda - \frac{(N-2)^{2}}{4}}} \in H^{1}(B_{\varepsilon}(0))$ is a radial solution to the equation \eqref{190}. We choose a constant $ C>0 $  which depends of $ \lambda,\varepsilon,\delta$ and N such that  $$ C \delta ^{-\frac{N-2}{2}+ \sqrt{ \lambda - \frac{(N-2)^{2}}{4}}} \geq \|v_{\varepsilon}\|_{L^{\infty}(\partial B_{\varepsilon}(0))}$$

Therefore, we get 
\begin{align}
    \left\{\begin{array}{cc} 
-\Delta (Cu_{\lambda} - v_{\varepsilon}) +\lambda\frac{ (C u_{\lambda} - v_{\varepsilon})}{|x|^{2}} = 0,  &		\textrm{ in} \quad  B_{\delta}(0) \\
	Cu_{\lambda}-v_{\varepsilon}\ge0, &  \textrm{ on }  \partial B_{\delta}(0),\\ 
	\end{array}\right.
\end{align}
Hence, using the proposition \ref{ Teorema 7} we obtain that $$ v_{\varepsilon}(x) \leq C u_{\lambda}(x),$$ for any $ x \in B_{\delta}(0)$ and $ \lambda \in (- \frac{(N-2)^{2}}{4},\infty) $. 

A direct computation shows that 
\begin{align}
    \int_{B_{\delta}(0)}\frac{v_{\varepsilon}^{2}}{|x|^{4}}dx \leq C\int_{B_{\delta}(0)}\frac{|x|^{-N+2+ 2\sqrt{ \lambda - \frac{(N-2)^{2}}{4}}}}{|x|^{4}}dx \leq C \int_{0}^{\delta}r^{-3+2\sqrt{\lambda - \frac{(N-2)^{2}}{4}}}dr.
\end{align}

This means that for $ \lambda >-\frac{N(N-4)}{4}$ we have 
\begin{align}\label{eq125}
  \int_{B_{\delta}(0)}\frac{v_{\varepsilon}^{2}}{|x|^{4}}dx <\infty  
\end{align}
We know that 
\begin{align}
  \left\{\begin{array}{ll} 
-\Delta v_{\varepsilon}=  f_{\varepsilon}-\lambda \frac{ v_{\varepsilon}}{|x|^{2}}  &		\textrm{ in} \quad   B_{\delta}(0)\\
v_{\varepsilon} \in H^{1}(B_{\delta}(0))
 	\end{array}\right.
\end{align}

By virtute of \eqref{eq125}  and $ f_{\varepsilon} \in L^{2}(B_{\delta}(0)) $ we deduce that $ f_{\varepsilon} - \lambda\frac{ v_{\varepsilon}}{|x|^{2}} \in L^{2}(B_{\delta}(0))$ and by standard elliptic regularity given by Calderon Zygmund we have $ v_{\varepsilon} \in H^{2}(B_{\delta}(0))$ for any $ \varepsilon>0$. 

It is obvious that we have  $ v_{\varepsilon} \in H^{2}(\Omega \setminus B_{\delta}(0))$ and therefore we get that $ v_{\varepsilon} \in H^{2}(\Omega)$. 

\textbf{Step 4}
We prove that $ v_{\varepsilon} \rightarrow v$ in $ H_{0}^{1}(\Omega)$

Indeed, we have by linearity of the problem that
$$ \int_{\Omega} \nabla (v_{\varepsilon}-v)\nabla\phi + \lambda \frac{v_{\varepsilon}-v}{|x|^{2}}\phi dx= \int_{\Omega} (f_{\varepsilon}-f)\phi dx,$$
for any $ \phi \in H^{1}_{0}(\Omega)$.

Taking $ \phi = v_{\varepsilon}-v$ and applying Hardy inequality and Cauchy-Schwartz inequality there exists $ C(\lambda, N)$ such that 
$$|\nabla(v_{\varepsilon}-v)|_{L^{2}(\Omega)}^{2} \leq C(\lambda,N)\|f_{\varepsilon}-f\|_{L^{2}(\Omega)}\|v_{\varepsilon}-v\|_{L^{2}(\Omega)}$$

By Poincar\'e inequality we obtain that 
$$ \|(v_{\varepsilon}-v)\|_{H^{1}_{0}(\Omega)}\leq C(\lambda,N)\|f_{\varepsilon}-f\|_{L^{2}(\Omega)}.$$
This means that $$ v_{\varepsilon} \rightarrow v$$ in $ L^{2}(\Omega)$ when $ \varepsilon \rightarrow 0$ because $ f_{\varepsilon} \rightarrow f$.

In fact we know that $ v_{\varepsilon} \rightarrow v$ a.e in $ \Omega$ up to a subsequence. 

\textbf{Step 5} We show that $v \in H^{2}(\Omega)$. 

By density we can extend the inequality \eqref{Ineq_lambda} to $ H_{0}^{2}(\Omega)$ when $ N \geq 5$. 
In fact we can't apply directly this result for $ v_{\varepsilon}$ since it does not belong in $ H_{0}^{2}(\Omega)$. 

We will use a cut-off argument. We take a function $ \eta \in C^{\infty}_{c}(\Omega)$, where $ 0 \leq \eta\leq 1$ such that 
\begin{align}\label{eq2b_alt}
 \eta(x):=  \left\{\begin{array}{ll} 
1,  &		\textrm{ for} \quad  |x|< \frac{1}{2}\\
	0, &  \textrm{ for }  |x|\geq 1,\\ 
	\end{array}\right.
\end{align} 

We evaluate 
\begin{align*}
\int_{\Omega}\frac{| v_{\varepsilon}|^{2}}{|x|^{4}}dx \leq \int_{\Omega} \frac{| v_{\varepsilon}\eta|^{2}}{|x|^{4}}dx + \int_{\Omega \setminus B_{\frac{1}{2}}(0)}\frac{|(v_{\varepsilon}(1-\eta))|^{2}}{|x|^{4}}dx  
\end{align*}

On the other hand, $ v_{\varepsilon} \eta \in H^{2}_{0}(\Omega)$ and we know that for $\lambda \in (\frac{-N^{2}+2N+2}{4},\infty)\setminus\{-\frac{N(N-4)}{4}\}$ we can apply \eqref{Ineq_lambda} which gives us that 
\begin{align*}
    \int_{\Omega}\frac{|v_{\varepsilon}\eta|^{2}}{|x|^{4}}dx \leq C(\lambda, N) \int_{\Omega}\left|-\Delta (v_{\varepsilon}\eta) + \lambda \frac{v_{\varepsilon}\eta}{|x|^{2}}\right|^{2}dx
\end{align*}
Putting $ v_{\varepsilon}\eta = v_{\varepsilon} + v_{\varepsilon}(\eta -1) $ we have that there exists $ C>0$ such that
\begin{align}\label{200}
    C \int_{\Omega}\left|-\Delta (v_{\varepsilon}\eta) + \lambda \frac{v_{\varepsilon}\eta}{|x|^{2}}\right|^{2}dx &\leq \int_{\Omega}\left|- \Delta v_{\varepsilon} + \lambda \frac{ v_{\varepsilon}}{|x|^{2}}\right|^{2}  dx +\\ \nonumber 
     & + \int_{\Omega \setminus B_{\frac{1}{2}(0)}}\left|- \Delta [v_{\varepsilon}(\eta-1)] + \lambda \frac{ v_{\varepsilon}(\eta -1)}{|x|^{2}} \right|^{2}dx \\ \nonumber  
     & \leq \int_{\Omega}\left|- \Delta v_{\varepsilon} + \lambda \frac{ v_{\varepsilon}}{|x|^{2}}\right|^{2}  dx + \int_{\Omega \setminus B_{\frac{1}{2}(0)} }\left|- \Delta v_{\varepsilon} + \lambda \frac{ v_{\varepsilon}}{|x|^{2}}\right|^{2}(\eta -1)^{2}dx \\ \nonumber 
     &+\int_{\Omega  \setminus B_{\frac{1}{2}}(0)}\left|-2\nabla v_{\varepsilon}\nabla \eta - \Delta \eta v_{\varepsilon}  \right|^{2}dx
\end{align}

Using Hardy inequality and Cauchy-Schwartz inequality we can estimate the last term of the previous inequality and we get that 

\begin{align*}
    \int_{\Omega \setminus B_{\frac{1}{2}}(0)} \left|-2\nabla v_{\varepsilon}\nabla \eta - \Delta \eta v_{\varepsilon}  \right|^{2}dx &\leq 4\|\nabla v_{\varepsilon}\|_{L^{2}(\Omega \setminus B_{\frac{1}{2}(0)})}\|\nabla \eta\|_{L^{\infty}(\Omega)}+ \|\Delta \eta\|_{L^{\infty}(\Omega)}\|v_{\varepsilon}\|_{L^{2}(\Omega \setminus B_\frac{1}{2})} \\
\end{align*}  
As in \textbf{step 1} we argue that $ \|v_{\varepsilon}\|_{L^{2}(\Omega \setminus B_{\frac{1}{2}}(0))} \leq \|f\|_{L^{2}(\Omega )}$, $ \|\nabla v_{\varepsilon}\|_{L^{2}(\Omega \setminus B_{\frac{1}{2}(0)})} \leq \|f\|_{L^{2}(\Omega)}$  we have that  
\begin{align*}
    \int_{\Omega \setminus B_{\frac{1}{2}}(0)} \left|-2\nabla v_{\varepsilon}\nabla \eta - \Delta \eta v_{\varepsilon}  \right|^{2}dx &\leq 4\|f\|_{L^{2}(\Omega )}\|\nabla \eta\|_{L^{\infty}(\Omega)}+ \|\Delta \eta\|_{L^{\infty}(\Omega)}\|f\|_{L^{2}(\Omega)} \\
\end{align*} 

Since $ - \Delta v_{\varepsilon} + \lambda \frac{ v_{\varepsilon}}{|x|^{2}} = f_{\varepsilon} $ the inequality \eqref{200} leads to 
$$\int_{\Omega}\left|-\Delta (v_{\varepsilon}\eta) + \lambda \frac{ v_{\varepsilon}\eta}{|x|^{2}}\right|^{2}dx \leq C(\lambda^{2}, |\nabla  \eta|_{L^{\infty}(\Omega)}, |\Delta \eta|_{L^{\infty}(\Omega)},N)\|f\|_{L^{2}(\Omega)},$$ where $C(\lambda^{2}, |\nabla  \eta|, |\Delta \eta|, N)>0$. 

Since $ \frac{1}{|x|^{4}}\leq 16$ for $ x \in \Omega \setminus B_\frac{1}{2}(0)$ we have that 
\begin{align*}
    \int_{\Omega \setminus B_{\frac{1}{2}}(0)}\frac{|v_{\varepsilon}(1-\eta)|^{2}}{|x|^{4}}dx&\leq \|(1-\eta)^{2}\|_{L^{\infty}(\Omega)}\int_{\Omega \setminus B_{\frac{1}{2}(0)}}\frac{| v_{\varepsilon}|^{2}}{|x|^{4}}dx  \\
    & \leq 16\|(1-\eta)^{2}\|_{L^{\infty}(\Omega)}\|f\|_{L^{2}(\Omega)}
\end{align*}

Putting all together we can estimate 
\begin{align*}
    \int_{\Omega}\frac{| v_{\varepsilon}|^{2}}{|x|^{4}}dx &\leq \int_{\Omega}\frac{| v_{\varepsilon}\eta|^{2}}{|x|^{4}}dx + \int_{\Omega \setminus B_{\frac{1}{2}(0)}}\frac{|v_{\varepsilon}(1-\eta)|^{2}}{|x|^{4}}dx \nonumber \\
    & \leq \tilde{C}(\lambda, N) \int_{\Omega}|f|^{2}dx 
\end{align*}

We proved that the integral $ \int_{\Omega}\frac{|v_{\varepsilon}|^{2}}{|x|^{4}}dx $ is uniformly bounded in $ \varepsilon$. 

By \textbf{step 4} Fatou lemma implies that 
$$ \int_{\Omega} \frac{| v|^{2}}{|x|^{4}}dx \leq \lim_{\varepsilon \rightarrow 0} \inf \int_{\Omega} \frac{| v_{\varepsilon}|^{2}}{|x|^{4}}  \leq \tilde{C}(\lambda,N) \int_{\Omega}|f|^{2}dx < \infty$$ 
This is sufficient to deduce that $ v \in H^{2}(B_{\delta}(0))$ and we obtain the conclusion. 
\end{proof}

\section*{\centering Acknowledgements}

The authors would like to thank David Krejčiřík for useful discussion on Kato's theory of relatively bounded perturbations, which provided a valuable perspective in the development of this work. We are also grateful to Tom ter Elst for pointing us to the reference \cite{Metafune2016}, which was useful in the context of the present paper. We sincerely appreciate their helpful suggestions and valuable insights.


\begin{thebibliography}{30}
\bibitem{Adams}
R. A. Adams and J. J. F. Fournier,
\emph{Sobolev Spaces},
2nd ed., Pure and Applied Mathematics, vol. 140,
Elsevier/Academic Press, Amsterdam, 2003.

\bibitem{Beckner2008}
W. Beckner,
\emph{On the Grushin Operator and Hyperbolic Symmetry},
Forum Math. \textbf{20} (2008), no.~4, 617--637.

\bibitem{anderson}
H. Brezis,
\emph{Functional Analysis, Sobolev Spaces and Partial Differential Equations},
Springer, New York, 2011.

\bibitem{CalderonZygmund1952}
A. P. Calder\'on and A. Zygmund,
\emph{On the Existence of Certain Singular Integrals},
Acta Mathematica \textbf{88} (1952), 85--139.

\bibitem{CalderonZygmund1956}
A. P. Calder\'on and A. Zygmund,
\emph{On Singular Integrals},
American Journal of Mathematics \textbf{78} (1956), no.~2, 289--309.

\bibitem{Cazacu1}
C. Cazacu,
\emph{The method of super-solutions in Hardy and Rellich type inequalities
in the $L^2$ setting: an overview of well-known results and short proofs},
Rev. Roumaine Math. Pures Appl. \textbf{66} (2021), no. 3--4, 617--638. 

\bibitem{Cazacu2}
C. Cazacu,
\emph{A new proof of the Hardy-Rellich inequality in any dimension},
Proc. Roy. Soc. Edinburgh Sect. A \textbf{150} (2020), no. 6, 2894--2904.

\bibitem{Cazacu3}
C. Cazacu and I. Fidel,
\emph{Weighted Hardy-Rellich type inequalities: improved best constants
and symmetry breaking},
Rev. Roumaine Math. Pures Appl. \textbf{69} (2024), no. 3--4, 397--413.

\bibitem{Dupaigne}
L. Dupaigne,
\emph{Stable Solutions of Elliptic Partial Differential Equations},
Chapman \& Hall/CRC Monographs and Surveys in Pure and Applied Mathematics,
vol. 143, Chapman \& Hall/CRC, Boca Raton, FL, 2011.

\bibitem{Charro}
L. C. Evans,
\emph{Partial Differential Equations},
2nd ed., Graduate Studies in Mathematics, vol. 19,
American Mathematical Society, Providence, RI, 2010.

\bibitem{Real-Ros-Oton}
X. Fern{\'a}ndez-Real and X. Ros-Oton,
\emph{Regularity Theory for Elliptic PDE},
Zurich Lectures in Advanced Mathematics, vol. 28,
EMS Press, Berlin, 2022.

\bibitem{Gesztesy}
F. Gesztesy and L. Littlejohn,
\emph{Factorizations and Hardy--Rellich-type inequalities},
in \emph{Nonlinear Partial Differential Equations, Mathematical Physics,
and Stochastic Analysis},
EMS Series of Congress Reports, 207--226,
European Mathematical Society, Zürich, 2018.

\bibitem{Xavier}
D. Gilbarg and N. S. Trudinger,
\emph{Elliptic Partial Differential Equations of Second Order},
2nd ed., Springer, Berlin, 2001.


\bibitem{Kato}
T. Kato,
\emph{Perturbation Theory for Linear Operators},
Reprint of the 1980 edition,
Springer-Verlag, Berlin, 1995.



\bibitem{Kim-Tsai}
H. Kim and T.-P. Tsai,
\emph{Existence, uniqueness, and regularity results for elliptic equations
with drift terms in critical weak spaces},
SIAM J. Math. Anal. \textbf{52} (2020), no. 2, 1146--1191.

\bibitem{Killip2018} R.~Killip, C.~Miao, M.~Visan, J.~Zhang, and J.~Zheng, \emph{Sobolev spaces adapted to the Schrödinger operator with inverse-square potential}, Math. Z. \textbf{288} (2018), no.~3--4, 1273--1298.

\bibitem{LP07}
T. Leonori and F. Petitta,
\emph{Existence and regularity results for some singular elliptic problems},
Adv. Nonlinear Stud. \textbf{7} (2007), no. 3, 329--344.

\bibitem{Metafune2016}
G. Metafune, N. Okazawa, M. Sobajima, and C. Spina,
\emph{Scale invariant elliptic operators with singular coefficients},
J. Evol. Equ. \textbf{16} (2016), no. 2, 391--439.

\bibitem{Metafune2015}
G. Metafune, M. Sobajima, and C. Spina,
\emph{Weighted Calder\'on--Zygmund and Rellich inequalities in $L^p$},
Math. Ann. \textbf{361} (2015), no. 1--2, 313--366.

\bibitem{Ireneo Peral}
I. Peral Alonso and F. Soria de Diego,
\emph{Elliptic and Parabolic Equations Involving the Hardy--Leray Potential},
De Gruyter Series in Nonlinear Analysis and Applications, vol. 38,
De Gruyter, Berlin--Boston, 2021.


\bibitem{Stampacchia}
G. Stampacchia,
\emph{Le problème de Dirichlet pour les équations elliptiques du second ordre
à coefficients discontinus},
Ann. Inst. Fourier (Grenoble) \textbf{15} (1965), no. 1, 189--258.









\end{thebibliography}

\end{document}